\documentclass{amsart}
\usepackage{amsfonts}

\newtheorem{theorem}{Theorem}[section]

\newtheorem{corollary}[theorem]{Corollary}
\newtheorem{lemma}[theorem]{Lemma}

\newtheorem{definition}[theorem]{Definition}
\newtheorem{proposition}[theorem]{Proposition}
\newtheorem{remark}[theorem]{Remark}
\newtheorem{comments }[theorem]{Comments}
\numberwithin{theorem}{section}
\input{tcilatex}

\begin{document}
\title[Multipliers and crossed products by Hilbert pro-$C^{\ast }$-bimodules]%
{Multipliers of Hilbert pro-$C^{\ast }$-bimodules and crossed products by
Hilbert pro-$C^{\ast }$-bimodules }
\author{Maria Joi\c{t}a, Radu-B. Munteanu and Ioannis Zarakas }
\address{Maria Joi\c{t}a \\
Department of Mathematics, Faculty of Applied Sciences, University
Politehnica of Bucharest, 313 Spl. Independentei Street, 060042, Bucharest,
Romania and Simion Stoilow Institute of Mathematics of the Roumanian
Academy, 21 Calea Grivitei Street, 010702, Bucharest, Romania}
\email{joita@fmi.unibuc.ro }
\urladdr{http://sites.google.com/a/g.unibuc.ro/maria-joita/}
\address{ Radu-B. Munteanu\\
Department of Mathematics, University of Bucharest, 14 Academiei St.,
010014, Bucharest, Romania}
\email{radu-bogdan.munteanu@g.unibuc.ro}
\address{Ioannis Zarakas, Department of Mathematics, University of Athens,
Panepistimiopolis, Athens 15784, Greece}
\email{gzarak@math.uoa.gr}

\begin{abstract}
In this paper we introduce the notion of multiplier of a Hilbert pro-$%
C^{\ast }$-bimodule and we investigate the structure of the multiplier
bimodule of a Hilbert pro-$C^{\ast }$-bimodule. We also investigate the
relationship between the crossed product $A\times _{X}\mathbb{Z}$ of a pro-$%
C^{\ast }$-algebra $A$ by a Hilbert pro-$C^{\ast }$-bimodule $X$ over $A$,
the crossed product $M(A)\times _{M(X)}\mathbb{Z}$ of the multiplier algebra 
$M(A)$ of $A$ by the multiplier bimodule $M(X)$ of $X$ and the multiplier
algebra $M(A\times _{X}\mathbb{Z})$ of $A\times _{X}\mathbb{Z}$.
\end{abstract}

\maketitle

\section{Introduction}

The notion of a Hilbert $C^{\ast }$-module is a generalization of that of a
Hilbert space in which the inner product takes its values in a $C^{\ast }$%
-algebra rather than in the field of complex numbers, but the theory of
Hilbert $C^{\ast }$-modules is different from the theory of Hilbert spaces
(for example, no every Hilbert $C^{\ast }$-submodule is complemented). In
1953, Kaplansky first used Hilbert $C^{\ast }$-modules over commutative $%
C^{\ast }$-algebras to prove that derivations of type $I$ $AW^{\ast }$%
-algebras are inner. In 1973, the theory was extended independently by
Paschke and Rieffel to non-commutative $C^{\ast }$-algebras and the latter
author used it to construct the theory of \textquotedblleft induced
representations of $C^{\ast }$-algebras\textquotedblright . Moreover,
Hilbert $C^{\ast }$-modules gave the right context for the extension of the
notion of Morita equivalence to $C^{\ast }$-algebras and have played a
crucial role in Kasparov's $KK$-theory. Finally, they may be considered as a
generalization of vector bundles to non-commutative $\ast $-algebras,
therefore they play a significant role in non-commutative geometry and, in
particular, in $C^{\ast }$-algebraic quantum group theory and groupoid $%
C^{\ast }$-algebras. The extension of such a rich in results concept, to the
case of pro-$C^{\ast }$-algebras could not be disregarded.

In \cite{Z}, Zarakas introduced the notion of a Hilbert pro-$C^{\ast }$%
-bimodule over a pro-$C^{\ast }$-algebra and studied its structure. In \cite%
{J2}, Joi\c{t}a investigated the structure of the multiplier module of a
Hilbert pro-$C^{\ast }$-module. In this paper we introduce the notion of
multiplier of a Hilbert pro-$C^{\ast }$-bimodule and we investigate the
structure of the multiplier bimodule of a Hilbert pro-$C^{\ast }$-bimodule.

In \cite{JZ}, Joi\c{t}a and Zarakas extended the construction of Abadie,
Eilers and Exel \cite{AEE} in context of pro-$C^{\ast }$-algebras and
associated to a Hilbert pro-$C^{\ast }$-bimodule $(X,A)$ a pro-$C^{\ast }$%
-algebra $A\times _{X}\mathbb{Z}$, called the crossed product of $A$ by $X$.
It is natural to ask what is the relationship between the pro-$C^{\ast }$%
-algebras associated to a Hilbert pro-$C^{\ast }$-bimodule $(X,A)$ and its
multiplier bimodule $(M(X),M(A))$.

The organization of this paper is as follows. In Section 2, we recall some
notations and definitions. Section 3 is devoted to investigate multipliers
of a Hilbert pro-$C^{\ast }$-bimodule. Given a Hilbert pro-$C^{\ast }$%
-bimodule $X$, we show that the Hilbert pro-$C^{\ast }$-bimodule structure
on $X$ extends to a Hilbert pro-$C^{\ast }$-bimodule structure on the
multiplier bimodule $M(X)$ of $X$. Also we define the strict topology on $%
M(X)$ and show that $X$ can be identified with a Hilbert pro-$C^{\ast }$%
-sub-bimodule of $M(X)$ which is dense in $M(X)$ with respect to the strict
topology. We introduce the notion of morphism of Hilbert pro-$C^{\ast }$%
-bimodules, and show that a nondegenerate morphism between Hilbert pro-$%
C^{\ast }$-bimodules is continuous with respect to the strict topology and
it extends to a unique morphism between the multiplier bimodules. Finally,
as in the case of Hilbert $C^{\ast }$-bimodules \cite{R}, we show that $%
(M(X),M(A))$ can be regarded as a maximal extension of $(X,A)$. Section 4 is
devoted to investigate the relationship between the crossed product $A\times
_{X}\mathbb{Z}$ of a pro-$C^{\ast }$-algebra $A$ by a Hilbert pro-$C^{\ast }$%
-bimodule $X$ over $A$, the crossed product $M(A)\times _{M(X)}\mathbb{Z}$
of the multiplier algebra $M(A)$ of $A$ by the multiplier bimodule $M(X)$of $%
X$ and the multiplier algebra $M(A\times _{X}\mathbb{Z})$ of $A\times _{X}%
\mathbb{Z}$. We show that the crossed product associated to a full Hilbert
pro-$C^{\ast }$-bimodule $(X,A)$ can be identified with a pro-$C^{\ast }$%
-subalgebra of the crossed product associated to $(M(X),M(A))$ and the
crossed product associated to $(M(X),M(A))$ can be identified with a pro-$%
C^{\ast }$-subalgebra of the multiplier algebra of the crossed product
associated to $\left( X,A\right) $. Crossed products by Hilbert pro-$C^{\ast
}$-bimodules are generalizations of crossed products of pro-$C^{\ast }$%
-algebras by inverse limit automorphism \cite{JZ}. As an application, we
prove that given an inverse limit automorphism $\alpha $ of a nonunital pro-$%
C^{\ast }$-algebra $A$, the crossed product of $M(A)$ by $\overline{\alpha }$%
, the extension of $\alpha \ $to $M(A)$, can be identified with a pro-$%
C^{\ast }$-subalgebra of the multiplier algebra $M(A\times _{\alpha }\mathbb{%
Z})$ of $A\times _{\alpha }\mathbb{Z}$.

\section{Preliminaries}

A complete Hausdorff topological $\ast $-algebra $A$\ whose topology is
given by a directed family of $C^{\ast }$-seminorms $\{p_{\lambda };\lambda
\in \Lambda \}\ \ $is called a \textit{pro-}$C^{\ast }$\textit{-algebra}.
Other terms used in the literature for pro-$C^{\ast }$-algebras are: locally 
$C^{\ast }$-algebras (A. Inoue, M. Fragoulopoulou, A. Mallios, etc.), $%
LMC^{\ast }$-algebras (G. Lassner, K. Schm\"{u}dgen), $b^{\ast }$-algebras
(C. Apostol).

Let $A$\ be a pro-$C^{\ast }$-algebra with the topology given by $\Gamma
=\{p_{\lambda };\lambda \in \Lambda \}$ and let $B$\ be a pro-$C^{\ast }$%
-algebra with the topology given by $\Gamma ^{\prime }=\{q_{\delta };\delta
\in \Delta \}$.

\textit{An approximate unit} of $A$ is a net $\{e_{i}\}_{i\in I}$ of
positive elements in $A$ such that $p_{\lambda }\left( e_{i}\right) \leq 1$
for all $i\in I$ and for all $\lambda \in \Lambda $ and the nets $%
\{e_{i}b\}_{i\in I}\ $and $\{be_{i}\}_{i\in I}$ converge to $b$ for all $%
b\in A$.

A \textit{pro-}$C^{\ast }$\textit{-morphism} is a continuous $\ast $%
-morphism $\varphi :A\rightarrow B$ (that is, $\varphi $ is linear, $\varphi
\left( ab\right) =\varphi (a)\varphi (b)$ and $\varphi (a^{\ast })=\varphi
(a)^{\ast }$ for all $a,b\in A$ and for each $q_{\delta }\in \Gamma ^{\prime
}$, there is $p_{\lambda }\in \Gamma $ such that $q_{\delta }\left( \varphi
(a)\right) \leq p_{\lambda }\left( a\right) $ for all $a\in A$). An
invertible pro-$C^{\ast }$-morphism $\varphi :A\rightarrow B$ is a pro-$%
C^{\ast }$-isomorphism if $\varphi ^{-1}$ is also pro-$C^{\ast }$-morphism.

If $\{A_{\lambda };\pi _{\lambda \mu }\}_{\lambda \geq \mu ,\lambda ,\mu \in
\Lambda }$ is an inverse system of $C^{\ast }$-algebras, then $%
\lim\limits_{\leftarrow \lambda }A_{\lambda }$ with the topology given by
the family of $C^{\ast }$-seminorms $\{p_{\lambda }\}_{\lambda \in \Lambda
}, $ with $p_{\lambda }\left( \left( a_{\mu }\right) _{\mu \in \Lambda
}\right) =\left\Vert a_{\lambda }\right\Vert _{A_{\lambda }}$ for all $%
\lambda \in \Lambda $, is a pro-$C^{\ast }$-algebra.

Let $A$\ be a pro-$C^{\ast }$-algebra with the topology given by $\Gamma
=\{p_{\lambda };\lambda \in \Lambda \}$. For $\lambda \in \Lambda $,\ $\ker
p_{\lambda }$\ is a closed $\ast $-bilateral ideal and $A_{\lambda }=A/\ker
p_{\lambda }$\ is a $C^{\ast }$-algebra in the $C^{\ast }$-norm $\left\Vert
\cdot \right\Vert _{p_{\lambda }}$\ induced by $p_{\lambda }$\ (that is, $%
\left\Vert a+\ker p_{{\small \lambda }}\right\Vert _{p_{\lambda }}=$\ $p_{%
{\small \lambda }}(a),$ for all $a\in A$). The canonical map from $A$ to $%
A_{\lambda }$ is denoted by $\pi _{\lambda }^{A},$ $\pi _{\lambda
}^{A}\left( a\right) =a+\ker p_{\lambda }$ for all $a\in A$. For $\lambda
,\mu \in \Lambda $\ with $\mu \leq \lambda $\ there is a surjective $C^{\ast
}$-morphism $\pi _{\lambda \mu }^{A}:A_{\lambda }\rightarrow A_{\mu }$\ such
that $\pi _{\lambda \mu }^{A}\left( a+\ker {\small p}_{\lambda }\right)
=a+\ker p_{\mu }$, and then $\{A_{\lambda };\pi _{\lambda \mu
}^{A}\}_{\lambda ,\mu \in \Lambda }$\ is an inverse system of $C^{\ast }$%
-algebras. Moreover, the pro-$C^{\ast }$-algebras$\ A$ and $%
\lim\limits_{\leftarrow \lambda }A_{\lambda }$ are isomorphic (Arens-Michael
decomposition). For more details we refer the reader to \cite{F,M,P}.

Here we recall some basic facts from \cite{J1} and \cite{Z} regarding
Hilbert pro-$C^{\ast }$-modules and Hilbert pro-$C^{\ast }$-bimodules
respectively.

Let $A$ be a pro-$C^{\ast }$-algebra whose topology is given by the family
of $C^{\ast }$-seminorms $\Gamma =\{p_{\lambda };\lambda \in \Lambda \}$.

\textit{A} \textit{right Hilbert pro-}$C^{\ast }$\textit{-module over }$A$%
\textit{\ }(or just \textit{Hilbert }$A$\textit{-module}), is a linear space 
$X$ that is also a right $A$-module equipped with a right $A$-valued inner
product $\left\langle \cdot ,\cdot \right\rangle _{A}$, that is $\mathbb{C}$%
- and $A$-linear in the second variable and conjugate linear in the first
variable, with the following properties:

\begin{enumerate}
\item $\left\langle x,x\right\rangle _{A}\geq 0$ and $\left\langle
x,x\right\rangle _{A}=0$ if and only if $x=0;$

\item $\left( \left\langle x,y\right\rangle _{A}\right) ^{\ast
}=\left\langle y,x\right\rangle _{A}$
\end{enumerate}

and which is complete with respect to the topology given by the family of
seminorms $\{p_{\lambda }^{A}\}_{\lambda \in \Lambda },$ with $p_{\lambda
}^{A}\left( x\right) =p_{\lambda }\left( \left\langle x,x\right\rangle
_{A}\right) ^{\frac{1}{2}},x\in X$. A Hilbert $A$-module $X$ is full if the
pro-$C^{\ast }$- subalgebra of $A$ generated by $\{\left\langle
x,y\right\rangle _{A};x,y\in X\}$ coincides with $A$.

\textit{A} \textit{left Hilbert pro-}$C^{\ast }$\textit{-module }$X$\ over\
a\ pro-$C^{\ast }$-algebra $A$ is defined in the same way, where for
instance the completeness is requested with respect to the family of
seminorms $\{^{A}p_{\lambda }\}_{\lambda \in \Lambda }$, where $%
^{A}p_{\lambda }\left( x\right) =p_{\lambda }\left( _{A}\left\langle
x,x\right\rangle \right) ^{\frac{1}{2}},x\in X$.

In case $X$ is a left Hilbert \text{pro-}$C^{\ast }$\text{-}module over $%
(A,\{p_{\lambda }\}_{\lambda \in \Lambda })$ and a right Hilbert \text{pro-}$%
C^{\ast }$\text{-}module over $(B,\{q_{\lambda }\}_{\lambda \in \Lambda })$,
such that the following relations hold:

\begin{itemize}
\item $_{A}\left\langle x,y\right\rangle z=x\left\langle y,z\right\rangle
_{B}$ for all $x,y,z\in X$ ,

\item $q_{\lambda }^{B}(ax)$ $\leq p_{\lambda }(a)q_{\lambda }^{B}\left(
x\right) $ and $^{A}p_{\lambda }(xb)$ $\leq q_{\lambda }(b)^{A}p_{\lambda
}\left( x\right) $ for all $x\in X,\,a\in A,\,b\in B$ and for all $\lambda
\in \Lambda $,
\end{itemize}

then we say that $X$ is \textit{a Hilbert }$A-B$\textit{\ pro-}$C^{\ast }$%
\textit{-bimodule}.

A Hilbert $A-B$ pro-$C^{\ast }$\text{-}bimodule $X$ is \textit{full} if it
is full as a right and as a left Hilbert pro-$C^{*}$-module.

Let $\Lambda $ be an upward directed set and $\{A_{\lambda };B_{\lambda
};X_{\lambda };\pi _{\lambda \mu };\chi _{\lambda \mu };\sigma _{\lambda \mu
};\lambda ,\mu \in \Lambda ,\lambda \geq \mu \}$ an inverse system of
Hilbert $C^{\ast }$-bimodules, that is:

\begin{itemize}
\item $\{A_{\lambda };\pi _{\lambda \mu };\lambda ,\mu \in \Lambda ,\lambda
\geq \mu \}$ and $\{B_{\lambda };\chi _{\lambda \mu };\lambda ,\mu \in
\Lambda ,\lambda \geq \mu \}$ are inverse systems of $C^{\ast }$-algebras;

\item $\{X_{\lambda };\sigma _{\lambda \mu };\lambda ,\mu \in \Lambda
,\lambda \geq \mu \}$ is an inverse system of Banach spaces;

\item for each $\lambda \in \Lambda ,$ $X_{\lambda }$ is a Hilbert $%
A_{\lambda }-B_{\lambda }$ $C^{\ast }$-bimodule;

\item $\left\langle \sigma _{\lambda \mu }\left( x\right) ,\sigma _{\lambda
\mu }\left( y\right) \right\rangle _{B_{\mu }}=\chi _{\lambda \mu }\left(
\left\langle x,y\right\rangle _{B_{\lambda }}\right) $ and $_{A_{\mu
}}\left\langle \sigma _{\lambda \mu }\left( x\right) ,\sigma _{\lambda \mu
}\left( y\right) \right\rangle =\pi _{\lambda \mu }\left( _{A_{\lambda
}}\left\langle x,y\right\rangle \right) $ for all $x,y\in X_{\lambda }$ and
for all $\lambda ,\mu \in \Lambda $ with $\lambda \geq \mu .$

\item $\sigma_{\lambda\mu}(x)\chi_{\lambda\mu}(b)=\sigma_{\lambda\mu}(xb),\,%
\,\pi_{\lambda\mu}(a)\sigma_{\lambda\mu}(x)=\sigma_{\lambda\mu}(ax)$ for all 
$x\in X_{\lambda},\,a\in\,A_{\lambda},\,b\in\,B_{\lambda}$ and for all $%
\lambda,\mu\in\Lambda$ such that $\lambda\geq\mu.$
\end{itemize}

Let $A=${{$\lim\limits_{\leftarrow \lambda }A_{\lambda }$, }}$B=$$%
\lim\limits_{\leftarrow \lambda }B_{\lambda }${\ and }$X=${{$%
\lim\limits_{\leftarrow \lambda }X_{\lambda }$. Then }}$X$ has a structure
of a Hilbert $A-B$ pro-$C^{\ast }$-bimodule with

\begin{center}
$\left( x_{\lambda }\right) _{\lambda \in \Lambda }\left( b_{\lambda
}\right) _{\lambda \in \Lambda }=\left( x_{\lambda }b_{\lambda }\right)
_{\lambda \in \Lambda }$ and $\left\langle \left( x_{\lambda }\right)
_{\lambda \in \Lambda },\left( y_{\lambda }\right) _{\lambda \in \Lambda
}\right\rangle _{B}$ $=\left( \left\langle x_{\lambda },y_{\lambda
}\right\rangle _{B_{\lambda }}\right) _{\lambda \in \Lambda }$

and

$\left( a_{\lambda }\right) _{\lambda \in \Lambda }\left( x_{\lambda
}\right) _{\lambda \in \Lambda }=\left( a_{\lambda }x_{\lambda }\right)
_{\lambda \in \Lambda }$ and $_{A}\left\langle \left( x_{\lambda }\right)
_{\lambda \in \Lambda },\left( y_{\lambda }\right) _{\lambda \in \Lambda
}\right\rangle =\left( _{A_{\lambda }}\left\langle x_{\lambda },y_{\lambda
}\right\rangle \right) _{\lambda \in \Lambda }$.
\end{center}

Let $X$ be a Hilbert $A-B$ pro-$C^{\ast }$-bimodule. Then, for each $\lambda
\in \Lambda ,$ $^{A}p_{\lambda }\left( x\right) =q_{\lambda }^{B}\left(
x\right) $ for all $x\in X$, and the normed space $X_{\lambda }=X/N_{\lambda
}^{B}$, where $N_{\lambda }^{B}=\{x\in X;q_{\lambda }^{B}\left( x\right)
=0\} $, is complete in the norm $||x+N_{\lambda }^{B}||_{X_{\lambda }}$ $%
=q_{\lambda }^{B}(x),x\in X$. Moreover, $X_{\lambda }$ has a canonical
structure of a Hilbert $A_{\lambda }-$ $B_{\lambda }$ $C^{\ast }$-bimodule
with $\left\langle x+N_{\lambda }^{B},y+N_{\lambda }^{B}\right\rangle
_{B_{\lambda }}$ $=\left\langle x,y\right\rangle _{B}+\ker q_{\lambda }$ and 
$_{A_{\lambda }}\left\langle x+N_{\lambda }^{B},y+N_{\lambda
}^{B}\right\rangle =_{A}\left\langle x,y\right\rangle +\ker p_{\lambda }$
for all $x,y\in X$. The canonical surjection from $X$ to $X_{\lambda }$ is
denoted by $\sigma _{\lambda }^{X}$. For $\lambda ,\mu \in \Lambda $ with $%
\lambda \geq \mu $, there is a canonical surjective linear map $\sigma
_{\lambda \mu }^{X}:X_{\lambda }\rightarrow X_{\mu }$ such that $\sigma
_{\lambda \mu }^{X}\left( x+N_{\lambda }^{B}\right) =x+N_{\mu }^{B}$ for all 
$x\in X$. Then $\{A_{\lambda };B_{\lambda };X_{\lambda };\pi _{\lambda \mu
}^{A};\sigma _{\lambda \mu }^{X};\pi _{\lambda \mu }^{B}\lambda ,\mu \in
\Lambda ,\lambda \geq \mu \}$ is an inverse system of Hilbert $C^{\ast }$%
-bimodules in the above sense.

Let $X$ and $Y$ be Hilbert pro-$C^{\ast }$-modules over $B$. A morphism $%
T:X\rightarrow Y$ of right modules is \textit{adjointable} if there is
another morphism of modules $T^{\ast }:Y\rightarrow X\ $such that $%
\left\langle Tx,y\right\rangle _{B}=\left\langle x,T^{\ast }y\right\rangle
_{B}$ for all $x\in X,y\in Y$. The vector space $L_{B}(X,Y)\ $of all
adjointable module morphisms from $X$ to $Y$ has a structure of locally
convex space under the topology given by the family of seminorms $%
\{q_{\lambda ,L_{B}(X,Y)}\}_{\lambda \in \Lambda }$, where $q_{\lambda
,L_{B}(X,Y)}\left( T\right) =\sup \{q_{\lambda }^{B}(Tx);x\in X,q_{\lambda
}^{B}\left( x\right) \leq 1\}$. Moreover, $\{L_{B_{\lambda }}(X_{\lambda
},Y_{\lambda });\chi _{\lambda \mu }^{L_{B}(X,Y)}$ $\lambda ,\mu \in \Lambda
,\lambda \geq \mu \}$ where $\chi _{\lambda \mu }^{L_{B}(X,Y)}:L_{B_{\lambda
}}(X_{\lambda },Y_{\lambda })\rightarrow L_{B_{\mu }}(X_{\mu },Y_{\mu })$ is
given by $\chi _{\lambda \mu }^{L_{B}(X,Y)}\left( T\right) \left( \sigma
_{\mu }^{X}\left( x\right) \right) =\sigma _{\lambda \mu }^{Y}\left(
T(\sigma _{\lambda }^{X}\left( x\right) )\right) $, is an inverse system of
Banach spaces and $L_{B}(X,Y)=${{$\lim\limits_{\leftarrow \lambda }$}} $%
L_{B_{\lambda }}(X_{\lambda },Y_{\lambda })$ up to an isomorphism of locally
convex spaces. The canonical projections $\chi _{\lambda }^{L_{B}(X,Y)}:$ $%
L_{B}(X,Y)\rightarrow $ $L_{B_{\lambda }}(X_{\lambda },Y_{\lambda }),$ $%
\lambda \in \Lambda $ are given by $\chi _{\lambda }^{L_{B}(X,Y)}\left(
T\right) \left( \sigma _{\lambda }^{X}\left( x\right) \right) =\sigma
_{\lambda }^{Y}\left( T(x)\right) $ for all $x\in X$. For $x\in X$ and $y\in
Y$, the map $\theta _{y,x}:X\rightarrow Y$ given by $\theta _{y,x}\left(
z\right) =y\left\langle x,z\right\rangle _{B}$ is an adjointable module
morphism and the closed subspace of $L_{B}(X,Y)$ generated by $\{\theta
_{y,x};x\in X$ and $y\in Y\}$ is denoted by $K_{B}(X,Y)$, whose elements are
usually called \textit{compact operators}. For $Y=X,$ $L_{B}(X)=L_{B}(X,X)$
is a pro-$C^{\ast }$-algebra with $\left( L_{B}(X)\right) _{\lambda
}=L_{B_{\lambda }}(X_{\lambda })$ for each $\lambda \in \Lambda $, and $%
K_{B}(X)=K_{B}(X,X)$ is a closed two-sided $\ast $-ideal of $L_{B}(X)$ with $%
\left( K_{B}(X)\right) _{\lambda }=K_{B_{\lambda }}(X_{\lambda })$ for each $%
\lambda \in \Lambda $.

A pro-$C^{\ast }$-algebra $A$ has a natural structure of Hilbert pro-$%
C^{\ast }$-module, and the multiplier algebra $M(A)$ has a structure of pro-$%
C^{\ast }$-algebra which is isomorphic to $L_{A}(A)$ \cite{P}. Moreover, pro-%
$C^{\ast }$-algebras $A$ and $K_{A}\left( A\right) $ are isomorphic and $A$
is a closed bilateral ideal of $M(A)$ which is dense in $M(A)$ with respect
to the strict topology. The strict topology on $M(A)$ is given by the family
of seminorms $\{p_{\left( \lambda ,a\right) }\}_{\left( \lambda ,a\right)
\in \Lambda \times A},$ where $p_{\left( \lambda ,a\right) }\left( b\right)
=p_{\lambda }\left( ab\right) +p_{\lambda }\left( ba\right) $ for all $b\in
M(A)$.

A pro-$C^{\ast }$-morphism $\varphi :A\rightarrow M(B)$\ is nondegenerate if 
$[\varphi \left( A\right) B]=B$, where $[\varphi \left( A\right) B]$ denotes
the closed subspace of $B$ generated by $\{\varphi \left( a\right) b;a\in
A,b\in B\}$. A nondegenerate pro-$C^{\ast }$-morphism $\varphi :A\rightarrow
M(B)$ extends to a unique pro-$C^{\ast }$-morphism $\overline{\varphi }%
:M(A)\rightarrow M(B)$ which is strictly continuous on bounded sets.

Throughout this paper, $A\ $and $B$ are two pro-$C^{\ast }$-algebras whose
topologies are given by the families of $C^{\ast }$-seminorms $\Gamma
=\{p_{\lambda };\lambda \in \Lambda \}$, respectively $\Gamma ^{\prime
}=\{q_{\delta };\delta \in \Delta \}.$

\section{Multipliers of Hilbert pro-$C^{\ast }$-bimodules}

Let $X$ and $Y$ be two Hilbert pro-$C^{\ast }$-modules over $A$.

\begin{proposition}
\label{bimodule}The vector space $L_{A}(X,Y)$ of all adjointable module maps
from $X$ to $Y$ has a natural structure of Hilbert $L_{A}(Y)-L_{A}(X)$ pro-$%
C^{\ast }$-bimodule with the bimodule structure given by 
\begin{equation*}
S\cdot T=S\circ T\text{ and }T\cdot R=T\circ R\ 
\end{equation*}%
for\ all\ $T\in L_{A}(X,Y),S\in L_{A}(Y)\ $and\ $R\in L_{A}(X)\ $and the
inner products given by 
\begin{equation*}
_{L_{A}(Y)}\left\langle T_{1},T_{2}\right\rangle =T_{1}\circ T_{2}^{\ast }%
\text{ and }\left\langle T_{1},T_{2}\right\rangle _{L_{A}(X)}=T_{1}^{\ast
}\circ T_{2}
\end{equation*}%
for all $T_{1},T_{2}\in L_{A}(X,Y)$.
\end{proposition}

\begin{proof}
It is a simple calculation to verify that $L_{A}(X,Y)$ has a structure of
pre-right Hilbert $L_{A}(X)$-pro-$C^{\ast }$-module with 
\begin{equation*}
T\cdot R=T\circ R\ \text{and\ }\left\langle T_{1},T_{2}\right\rangle
_{L_{A}(X)}=T_{1}^{\ast }\circ T_{2}\ 
\end{equation*}%
and $L_{A}(X,Y)$ has a structure of pre-left Hilbert $L_{A}(Y)$-pro-$C^{\ast
}$-module with 
\begin{equation*}
S\cdot T=S\circ T\ \text{and}_{\ L_{A}(Y)}\left\langle
T_{1},T_{2}\right\rangle =T_{1}\circ T_{2}^{\ast }.
\end{equation*}%
Moreover,%
\begin{eqnarray*}
p_{\lambda }^{L_{A}(X)}\left( T\right) ^{2} &=&p_{\lambda ,L_{A}(X)}\left(
\left\langle T,T\right\rangle _{L_{A}(X)}\right) =p_{\lambda
,L_{A}(X)}\left( T^{\ast }\circ T\right) \\
&=&\left\Vert \chi _{_{\lambda }}^{L_{A}(X,Y)}\left( T\right) ^{\ast }\chi
_{_{\lambda }}^{L_{A}(X,Y)}\left( T\right) \right\Vert _{L_{A_{\lambda
}}(X_{\lambda })} \\
&&\text{(see, for example, the proof of Proposition 1.10 \cite{EKQR})} \\
&=&\left\Vert \chi _{_{\lambda }}^{L_{A}(X,Y)}\left( T\right) \right\Vert
_{L_{A_{\lambda }}(X_{\lambda },Y_{\lambda })}^{2}=p_{\lambda
.L_{A}(X,Y)}\left( T\right) ^{2}
\end{eqnarray*}%
and%
\begin{eqnarray*}
^{L_{A}(Y)}p_{\lambda }\left( T\right) ^{2} &=&p_{\lambda ,L_{A}(Y)}\left(
_{L_{A}(Y)}\left\langle T,T\right\rangle \right) =p_{\lambda
,L_{A}(Y)}\left( T\circ T^{\ast }\right) \\
&=&\left\Vert \chi _{_{\lambda }}^{L_{A}(X,Y)}\left( T\right) \chi
_{_{\lambda }}^{L_{A}(X,Y)}\left( T\right) ^{\ast }\right\Vert
_{L_{A_{\lambda }}(Y_{\lambda })} \\
&=&\left\Vert \chi _{_{\lambda }}^{L_{A}(X,Y)}\left( T\right) ^{\ast
}\right\Vert _{L_{A_{\lambda }}(Y_{\lambda },X_{\lambda })}^{2} \\
&&\text{(see, for example, the proof of Proposition 1.10 \cite{EKQR})} \\
&=&\left\Vert \chi _{_{\lambda }}^{L_{A}(X,Y)}\left( T\right) \right\Vert
_{L_{A_{\lambda }}(X_{\lambda },Y_{\lambda })}^{2}=p_{\lambda
.L_{A}(X,Y)}\left( T\right) ^{2}
\end{eqnarray*}%
for all $T\in L_{A}(X,Y)\ $and for all $\lambda \in \Lambda $. Therefore, $%
L_{A}(X,Y)$ is a left Hilbert $L_{A}(Y)$-module and a right Hilbert $%
L_{A}(X) $-module.

Also it is easy to check that $_{L_{A}(Y)}\left\langle
T_{1},T_{2}\right\rangle \cdot T_{3}=T_{1}\cdot \left\langle
T_{2},T_{3}\right\rangle _{L_{A}(X)}$ for all $T_{1},T_{2},T_{3}\in
L_{A}(X,Y)$, and since $p_{\lambda }^{L_{A}(X)}\left( T\right) =$ $%
^{L_{A}(Y)}p_{\lambda }\left( T\right) =p_{\lambda .L_{A}(X,Y)}\left(
T\right) $ for all $T$ $\in L_{A}(X,Y)\ $and for all $\lambda \in \Lambda $, 
$L_{A}(X,Y)$ has a structure of Hilbert $L_{A}(Y)-L_{A}(X)$ pro-$C^{\ast }$%
-bimodule.
\end{proof}

\begin{remark}
Suppose that $\left( X,A\right) $ is a full\textbf{\ }Hilbert pro-$C^{\ast }$%
-bimodule. Then there is a pro-$C^{\ast }$-isomorphism $\Phi
_{A}:A\rightarrow K_{A}(X)$ given by $\Phi _{A}\left( a\right) \left(
x\right) =a\cdot x$ which extends to a pro-$C^{\ast }$-isomorphism $%
\overline{\Phi _{A}}:M(A)\rightarrow L_{A}(X)$. Moreover, $p_{\lambda
,L_{A}(X)}\left( \Phi _{A}\left( a\right) \right) =p_{\lambda }\left(
a\right) \ $for\ all $a\in A$ and $\lambda \in \Lambda $. Identifying $M(A)$
with $L_{A}(A)$ and using Proposition \ref{bimodule} and \cite[Proposition
2.5]{R}, we obtain a natural structure of Hilbert $M(A)-M(A)$ pro--$C^{\ast
} $-bimodule on $L_{A}(A,X)$ with 
\begin{equation*}
m\cdot T=\overline{\Phi _{A}}\left( m\right) \circ T\text{ and }%
_{M(A)}\left\langle T_{1},T_{2}\right\rangle =\overline{\Phi _{A}^{-1}}%
\left( T_{1}\circ T_{2}^{\ast }\right)
\end{equation*}%
and 
\begin{equation*}
T\cdot m=T\circ m\text{ and }\left\langle T_{1},T_{2}\right\rangle
_{M(A)}=T_{1}^{\ast }\circ T_{2}
\end{equation*}%
for all $T,T_{1},T_{2}\in L_{A}(A,X)$ and $m\in M(A)$.
\end{remark}

\begin{definition}
Let $\left( X,A\right) $ be a full Hilbert pro-$C^{\ast }$-bimodule. The
Hilbert $M(A)-M(A)$ pro-$C^{\ast }$-bimodule $L_{A}(A,X)$ is called the
multiplier bimodule of $X$ and it is denoted by $M(X)$.
\end{definition}

The following definition is a generalization of \cite[Definition\ 1.25]{EKQR}%
.

\begin{definition}
The strict topology on $M(X)$ is given by the family of seminorms $%
\{p_{\left( \lambda ,a\right) }\}_{\left( \lambda ,a\right) \in \Lambda
\times A}$, where $p_{\left( \lambda ,a\right) }\left( T\right) =p_{\lambda
}^{M(A)}\left( T\cdot a\right) +\ p_{\lambda }^{M(A)}\left( a\cdot T\right)
\ $ for all $T\in M(X)$ and $a\in A$.
\end{definition}

\begin{remark}
Let $\{T_{n}\}_{n}$ be a sequence in $M(X).$

\begin{enumerate}
\item If $\{T_{n}\}_{n}$ is strictly convergent, then it is bounded. Indeed,
if $\{T_{n}\}_{n}$ converges strictly to $T$, then for each $\lambda \in
\Lambda $, since%
\begin{eqnarray*}
\left\Vert \chi _{\lambda }^{M\left( X\right) }\left( T_{n}\right) \pi
_{\lambda }^{A}\left( a\right) -\chi _{\lambda }^{M\left( X\right) }\left(
T\right) \pi _{\lambda }^{A}\left( a\right) \right\Vert _{X_{\lambda }}
&=&p_{\lambda }^{A}\left( T_{n}\left( a\right) -T\left( a\right) \right) \\
&=&p_{\lambda }^{M(A)}\left( T_{n}\cdot a-T\cdot a\right) ,
\end{eqnarray*}%
the sequence $\{\chi _{\lambda }^{M(X)}\left( T_{n}\right) \pi _{\lambda
}^{A}\left( a\right) \}_{n}$ converges to $\chi _{\lambda }^{M(X)}\left(
T\right) \pi _{\lambda }^{A}\left( a\right) $ for all $a\in A$ and by the
Banach-Steinhaus theorem there is $M_{\lambda }>0$ such that 
\begin{equation*}
p_{\lambda }^{M(A)}\left( T_{n}\right) =p_{\lambda ,L_{A}(A,X)}\left(
T_{n}\right) =\left\Vert \chi _{\lambda }^{M(X)}\left( T_{n}\right)
\right\Vert _{L_{A_{\lambda }}(A_{\lambda },X_{\lambda })}\leq M_{\lambda }.
\end{equation*}

\item If $\{T_{n}\}_{n}$ converges strictly to $0$, then the sequences $%
\{\left\langle T_{n},T_{n}\right\rangle _{M(A)}\}_{n}$ and $%
\{_{M(A)}\left\langle T_{n},T_{n}\right\rangle \}_{n}$ are strictly
convergent to $0$ in $M(A)$.
\end{enumerate}
\end{remark}

Suppose that $X$ is a Hilbert pro-$C^{\ast }$-module over $A$. In \cite[%
Definition 3.2]{J2}, the strict topology on $L_{A}(A,X)$ is given by the
family of seminorms $\{p_{\left( \lambda ,a,x\right) }\}_{\left( \lambda
,a,x\right) \in \Lambda \times A\times X}$, where $p_{\left( \lambda
,a,x\right) }(T)=p_{\lambda }^{A}\left( T(a)\right) +p_{\lambda }\left(
T^{\ast }(x)\right) $. We will show that this definition coincides with the
above definition of the strict topology on $M(X)$ on bounded subsets when $%
X\ $ is a full Hilbert $A-A$ pro-$C^{\ast }$-bimodule. To show this, we will
use the following result.

\begin{lemma}
\label{factor}Let $X$ be a Hilbert pro-$C^{\ast }$-module over $A$. For each 
$x$ in $X$ there is a unique element $y$ in $X$ such that $x=y\left\langle
y,y\right\rangle _{A}.$
\end{lemma}

\begin{proof}
Let $x\in X$. For each $\lambda \in \Lambda $, there is a unique element $%
y_{\lambda }\in X_{\lambda }$ such that $\sigma _{\lambda }^{X}\left(
x\right) =y_{\lambda }\left\langle y_{\lambda },y_{\lambda }\right\rangle
_{A_{\lambda }}$(see, for example, \cite[Proposition 2.31]{RW}). Let $%
\lambda ,\mu \in \Lambda $ with $\lambda \geq \mu .$ From 
\begin{equation*}
\sigma _{\mu }^{X}\left( x\right) =\sigma _{\lambda \mu }^{X}(\sigma
_{\lambda }^{X}\left( x\right) )=\sigma _{\lambda \mu }^{X}\left( y_{\lambda
}\right) \left\langle \sigma _{\lambda \mu }^{X}\left( y_{\lambda }\right)
,\sigma _{\lambda \mu }^{X}\left( y_{\lambda }\right) \right\rangle _{A_{\mu
}}
\end{equation*}%
and \cite[Proposition 2.31]{RW}, we deduce that $\sigma _{\lambda \mu
}^{X}\left( y_{\lambda }\right) =y_{\mu }.$ Therefore, there exists $y\in X$
such that $\sigma _{\lambda }^{X}\left( y\right) =y_{\lambda }$ for all $%
\lambda \in \Lambda $ and $x=y\left\langle y,y\right\rangle _{A}.$ Moreover, 
$y$ is unique with this property.
\end{proof}

\begin{proposition}
\label{bounded}Let $(X,A)$ be a full Hilbert pro-$C^{\ast }$-bimodule and $%
\{T_{i}\}_{i\in I}$ a net in $M(X)$.

\begin{enumerate}
\item If $\{T_{i}\}_{i\in I}$ converges strictly to $0$,$\ $then $%
\{p_{\left( \lambda ,a,x\right) }(T_{i})\}_{i\in I}$ converges to $0$ for
all $a\in A$, for all $x\in X$ and for all $\lambda \in \Lambda .$

\item If $\{T_{i}\}_{i\in I}$ is bounded and $\{p_{\left( \lambda
,a,x\right) }(T_{i})\}_{i\in I}$ converges to $0$ for all $a\in A$, for all $%
x\in X$ and for all $\lambda \in \Lambda ,$ then $\{T_{i}\}_{i\in I}$
converges strictly to $0.$
\end{enumerate}
\end{proposition}

\begin{proof}
(1) If the net $\{T_{i}\}_{i\in I}$ converges strictly to $0$, then $%
\{p_{\lambda }^{A}\left( T_{i}(a)\right) \}_{i\in I}$ converges to $0$ for
all $a\in A$ and $\lambda \in \Lambda $. Let $x\in X$ and $\lambda \in
\Lambda $. Then, by Lemma \ref{factor}, there is $y\in X$ such that $%
x=y\left\langle y,y\right\rangle _{A}=\theta _{y,y}\left( y\right) $. From%
\begin{eqnarray*}
p_{\lambda }\left( T_{i}^{\ast }\left( x\right) \right) &=&p_{\lambda
}\left( T_{i}^{\ast }\left( \theta _{y,y}\left( y\right) \right) \right)
\leq p_{\lambda ,L_{A}(X,A)}\left( T_{i}^{\ast }\circ \theta _{y,y}\right)
p_{\lambda }^{A}\left( y\right) \\
&=&p_{\lambda ,L_{A}(A,X)}\left( \theta _{y,y}\circ T_{i}\right) p_{\lambda
}^{A}\left( y\right) =p_{\lambda }^{M(A)}\left( \theta _{y,y}\circ
T_{i}\right) p_{\lambda }^{A}\left( y\right) ,
\end{eqnarray*}%
we deduce that the net $\{p_{\lambda }\left( T_{i}^{\ast }\left( x\right)
\right) \}_{i\in I}$ converges to $0$.

(2) If $\{p_{\left( \lambda ,a,x\right) }\left( T_{i}\right) \}_{i\in I}$
converges to $0$ for all $a\in A,$ $x\in X$ and $\lambda \in \Lambda $, then 
$\{p_{\lambda }^{A}\left( T_{i}(a)\right) \}_{i\in I}$ converges to $0$ for
all $a\in A$ and $\lambda \in \Lambda $. Let $S\in K_{A}(X)$, $\lambda \in
\Lambda $ and $\varepsilon >0$. Then there is $\tsum\limits_{k=1}^{n}\theta
_{x_{k},y_{k}}$ such that $p_{\lambda ,L_{A}(X)}\left(
S-\tsum\limits_{k=1}^{n}\theta _{x_{k},y_{k}}\right) <\varepsilon ,$ and
since $\{T_{i}\}_{i\in I}$ is bounded, there is $M_{\lambda }>0$ such that $%
p_{\lambda }^{M(A)}\left( T_{i}\right) \leq M_{\lambda }$ for all $i\in I$.
From 
\begin{eqnarray*}
p_{\lambda }^{M(A)}\left( S\circ T_{i}\right) &\leq &p_{\lambda
,L_{A}(X)}\left( S-\tsum\limits_{k=1}^{n}\theta _{x_{k},y_{k}}\right) \
p_{\lambda }^{M(A)}\left( T_{i}\right) +p_{\lambda }^{M(A)}\left(
\tsum\limits_{k=1}^{n}\theta _{x_{k},y_{k}}\circ T_{i}\right) \\
&\leq &\varepsilon M_{\lambda }+p_{\lambda ,L_{A}(A,X)}\left(
\tsum\limits_{k=1}^{n}\theta _{x_{k},T_{i}^{\ast }\left( y_{k}\right)
}\right) \\
&\leq &\varepsilon M_{\lambda }+\tsum\limits_{k=1}^{n}p_{\lambda }^{A}\left(
x_{k}\right) p_{\lambda }\left( T_{i}^{\ast }\left( y_{k}\right) \right)
\end{eqnarray*}%
we deduce that $\{p_{\lambda }^{M(A)}\left( S\circ T_{i}\right) \}_{i\in I}$
converges to $0$.
\end{proof}

Let $\left( X,A\right) $ and $\left( Y,B\right) $ be two Hilbert pro-$%
C^{\ast }$-bimodules.

\begin{definition}
\textit{A morphism of Hilbert pro-}$C^{\ast }$\textit{-bimodules from }$%
\left( X,A\right) $ to $\left( Y,B\right) $ is a pair $\left( \Phi ,\varphi
\right) $ consisting of a pro-$C^{\ast }$-morphism $\varphi :A\rightarrow B$
and a map $\Phi :X\rightarrow Y$ such that:

\begin{enumerate}
\item $\Phi \left( xa\right) =\Phi \left( x\right) \varphi \left( a\right) $
for all $x\in X$ and for all $a\in A;$

\item $\Phi \left( ax\right) =\varphi \left( a\right) \Phi \left( x\right) $
for all $x\in X$ and for all $a\in A;$

\item $\left\langle \Phi \left( x\right) ,\Phi \left( y\right) \right\rangle
_{B}=\varphi \left( \left\langle x,y\right\rangle _{A}\right) $ for all $%
x,y\in X;$

\item $_{B}\left\langle \Phi \left( x\right) ,\Phi \left( y\right)
\right\rangle =\varphi \left( _{A}\left\langle x,y\right\rangle \right) $
for all $x,y\in X$.
\end{enumerate}
\end{definition}

{The relation $\left( 3\right) $ implies the relation $(1)$ and the relation 
$(4)$ implies $(2)$.}

{If }$\left( \Phi ,\varphi \right) :${\ }$\left( X,A\right) $ $\rightarrow
\left( Y,B\right) $ {is a morphism of Hilbert pro-$C^{\ast }$-bimodules,
then }$\Phi ${\ is continuous, since for each }$\delta \in \Delta ,$ there
is $\lambda \in \Lambda $ such that%
\begin{equation*}
q_{\delta }^{B}\left( \Phi \left( x\right) \right) ^{2}=q_{\delta }\left(
\left\langle \Phi \left( x\right) ,\Phi \left( x\right) \right\rangle
_{B}\right) =q_{\delta }\left( \varphi \left( \left\langle x,x\right\rangle
_{A}\right) \right) \leq p_{\lambda }\left( \left\langle x,x\right\rangle
_{A}\right) =p_{\lambda }^{A}\left( x\right) ^{2}
\end{equation*}%
for all $x\in X${. }It is easy to check that if $\varphi $ is injective,
then $\Phi $ is injective, and if $\left( X,A\right) $ is full and $\Phi $
is injective, then $\varphi $ is injective.

\begin{definition}
An isomorphism of \textit{Hilbert pro-}$C^{\ast }$\textit{-bimodules is a
morphism of Hilbert pro-}$C^{\ast }$\textit{-bimodules }$\left( \Phi
,\varphi \right) $ such that $\varphi $ is a pro-$C^{\ast }$-isomorphism and
the map $\Phi $ is bijective.

\textit{Hilbert pro-}$C^{\ast }$\textit{-bimodules }$\left( X,A\right) $%
\textit{\ and }$\left( Y,B\right) $ are \textit{isomorphic }if there is an
isomorphism of Hilbert pro-$C^{\ast }$-bimodules $\left( \Phi ,\varphi
\right) :\left( X,A\right) \rightarrow \left( Y,B\right) $.
\end{definition}

\begin{definition}
A \textit{morphism of Hilbert pro-}$C^{\ast }$\textit{-bimodules }{$\left(
\Phi ,\varphi \right) :$}$\left( X,A\right) $ $\rightarrow $ $\left(
M(Y),M(B)\right) ${\ is nondegenerate if }$\varphi $ is nondegenerate and $%
\left[ \Phi \left( X\right) B\right] =Y$.
\end{definition}

\begin{remark}
If {$\left( \Phi ,\varphi \right) :$}$\left( X,A\right) $ $\rightarrow $ $%
\left( M(Y),M(B)\right) ${\ is nondegenerate and }$\left( X,A\right) $ is
full{, then $\left( \Phi ,\varphi \right) $ is nondegenerate in the sense of 
\cite[Definition 3.1]{J3}, since}%
\begin{eqnarray*}
\left[ \Phi \left( X\right) ^{\ast }Y\right] &=&\left[ \Phi \left( X\right)
^{\ast }\Phi \left( X\right) B\right] =\left[ \left\langle \Phi \left(
X\right) ,\Phi \left( X\right) \right\rangle _{M(B)}B\right] \\
&=&\left[ \varphi \left( \left\langle X,X\right\rangle _{A}\right) B\right] =%
\left[ \varphi \left( A\right) B\right] =B.
\end{eqnarray*}
\end{remark}

\begin{lemma}
\label{strict1}Let $(X,A)$ be a full Hilbert pro-$C^{\ast }$-bimodule. Then
the maps 
\begin{equation*}
\left( \chi _{\lambda }^{L_{A}(A,X)},\pi _{\lambda }^{M(A)}\right) :\left(
M(X),M(A)\right) \rightarrow \left( M(X_{\lambda }),M(A_{\lambda })\right)
,\lambda \in \Lambda ,
\end{equation*}%
where $\pi _{\lambda }^{M(A)}=\chi _{\lambda }^{L_{A}(A)},$ and 
\begin{equation*}
\left( \chi _{\lambda \mu }^{L_{A}(A,X)},\pi _{\lambda \mu }^{M(A)}\right)
:\left( M(X_{\lambda }),M(A_{\lambda })\right) \rightarrow \left( M(X_{\mu
}),M(A_{\mu })\right) ,\lambda ,\mu \in \Lambda \text{ with }\lambda \geq \mu
\end{equation*}%
where$\ \pi _{\lambda \mu }^{M(A)}=\chi _{\lambda \mu }^{L_{A}(A)}$, are all
strictly continuous morphisms of Hilbert bimodules.
\end{lemma}

\begin{proof}
Let $\lambda ,\mu \in \Lambda $ with $\lambda \geq \mu $. For $%
T_{1},T_{2}\in M(X_{\lambda })$ we have 
\begin{eqnarray*}
\left\langle \chi _{\lambda \mu }^{L_{A}(A,X)}\left( T_{1}\right) ,\chi
_{\lambda \mu }^{L_{A}(A,X)}\left( T_{2}\right) \right\rangle _{M(A_{\mu })}
&=&\chi _{\lambda \mu }^{L_{A}(A,X)}\left( T_{1}\right) ^{\ast }\circ \chi
_{\lambda \mu }^{L_{A}(A,X)}\left( T_{2}\right) \\
&=&\chi _{\lambda \mu }^{L_{A}(X,A)}\left( T_{1}^{\ast }\right) \circ \chi
_{\lambda \mu }^{L_{A}(A,X)}\left( T_{2}\right) \\
&=&\chi _{\lambda \mu }^{L_{A}(A)}\left( T_{1}^{\ast }\circ T_{2}\right)
=\pi _{\lambda \mu }^{M(A)}\left( \left\langle T_{1},T_{2}\right\rangle
_{M(A_{\lambda })}\right)
\end{eqnarray*}%
and 
\begin{eqnarray*}
_{M(A_{\mu })}\left\langle \chi _{\lambda \mu }^{L_{A}(A,X)}\left(
T_{1}\right) ,\chi _{\lambda \mu }^{L_{A}(A,X)}\left( T_{2}\right)
\right\rangle &=&\overline{\Phi _{A_{\mu }}^{-1}}\left( \chi _{\lambda \mu
}^{L_{A}(A,X)}\left( T_{1}\right) \circ \chi _{\lambda \mu
}^{L_{A}(A,X)}\left( T_{2}\right) ^{\ast }\right) \\
&=&\overline{\Phi _{A_{\mu }}^{-1}}\left( \chi _{\lambda \mu
}^{L_{A}(A,X)}\left( T_{1}\right) \circ \chi _{\lambda \mu
}^{L_{A}(X,A)}\left( T_{2}^{\ast }\right) \right) \\
&=&\overline{\Phi _{A_{\mu }}^{-1}}\left( \chi _{\lambda \mu
}^{L_{A}(X)}\left( T_{1}\circ T_{2}^{\ast }\right) \right) \\
&=&\chi _{\lambda \mu }^{L_{A}(A)}\left( \overline{\Phi _{A_{\lambda }}^{-1}}%
\left( T_{1}\circ T_{2}^{\ast }\right) \right) \\
&=&\pi _{\lambda \mu }^{M(A)}\left( _{M(A_{\lambda })}\left\langle
T_{1},T_{2}\right\rangle \right) .
\end{eqnarray*}%
Therefore, $\left( \chi _{\lambda \mu }^{L_{A}(A,X)},\pi _{\lambda \mu
}^{M(A)}\right) $ is a morphism of Hilbert $C^{\ast }$-bimodules.

Let $\{T_{i}\}_{i\in I}$ be a net in $M(X_{\lambda })$ which converges
strictly to $0$. From%
\begin{eqnarray*}
\left\Vert \chi _{\lambda \mu }^{L_{A}(A,X)}\left( T_{i}\right) \pi _{\mu
}^{A}\left( a\right) \right\Vert _{X_{\mu }} &=&\left\Vert \chi _{\lambda
\mu }^{L_{A}(A,X)}\left( T_{i}\right) \pi _{\lambda \mu }^{M(A)}\left( \pi
_{\lambda }^{A}\left( a\right) \right) \right\Vert _{X_{\mu }} \\
&=&\left\Vert \sigma _{\lambda \mu }^{X}\left( T_{i}\left( \pi _{\lambda
}^{A}\left( a\right) \right) \right) \right\Vert _{X_{\mu }}\leq \left\Vert
T_{i}\left( \pi _{\lambda }^{A}\left( a\right) \right) \right\Vert
_{X_{\lambda }}
\end{eqnarray*}%
for all $a\in A$, and%
\begin{eqnarray*}
\left\Vert \chi _{\mu }^{L_{A}(X)}\left( S\right) \circ \chi _{\lambda \mu
}^{L_{A}(A,X)}\left( T_{i}\right) \right\Vert _{M(X_{\mu })} &=&\left\Vert
\chi _{\lambda \mu }^{L_{A}(X)}\left( \chi _{\lambda }^{L_{A}(X)}\left(
S\right) \right) \circ \chi _{\lambda \mu }^{L_{A}(A,X)}\left( T_{i}\right)
\right\Vert _{L_{A_{\mu }}(A_{\mu },X_{\mu })} \\
&=&\left\Vert \chi _{\lambda \mu }^{L_{A}(A,X)}\left( \chi _{\lambda
}^{L_{A}(X)}\left( S\right) \circ T_{i}\right) \right\Vert _{L_{A_{\mu
}}(A_{\mu },X_{\mu })} \\
&\leq &\left\Vert \chi _{\lambda }^{L_{A}(X)}\left( S\right) \circ
T_{i}\right\Vert _{L_{A_{\lambda }}(A_{\lambda },X_{\lambda })}
\end{eqnarray*}%
for all$\ S\in K_{A}(X)$, and taking into account that $K_{A_{\mu }}(X_{\mu
})=\chi _{\mu }^{L_{A}(X)}\left( K_{A}(X)\right) $, we deduce that the net $%
\{\chi _{\lambda \mu }^{M(X)}\left( T_{i}\right) \}_{i\in I}$ converges
strictly to $0$.

In a similar way, we show that the maps $\left( \chi _{\lambda
}^{L_{A}(A,X)},\pi _{\lambda }^{M(A)}\right) :\left( M(X),M(A)\right)
\rightarrow \left( M(X_{\lambda }),M(A_{\lambda })\right) ,\lambda \in
\Lambda $ are all strictly continuous morphisms of Hilbert bimodules.
\end{proof}

\begin{theorem}
\label{multiplier}Let $\left( X,A\right) $ be a full Hilbert pro-$C^{\ast }$%
-bimodule.

\begin{enumerate}
\item $M(X)$ is complete with respect to the strict topology;

\item $\left( \iota _{X},\iota _{A}\right) :\left( X,A\right) \rightarrow
\left( M(X),M(A)\right) $, where $\iota _{X}\left( x\right) \left( a\right)
=xa$ and $\iota _{A}\left( b\right) \left( a\right) =ba$ for all $x\in X$
and $a,b\in A$, is a nondegenerate morphism of Hilbert pro-$C^{\ast }$%
-bimodules;

\item $X$ can be identified with a closed $M(A)-M(A)$ pro-$C^{\ast }$%
-sub-bimodule of $M(X)$ which is dense in $M(X)$ with respect to the strict
topology.
\end{enumerate}
\end{theorem}

\begin{proof}
(1) For each $\lambda \in \Lambda ,$ $M(X_{\lambda })$ has a structure of
Hilbert $M(A_{\lambda })-$ $M(A_{\lambda })\ \ C^{\ast }$-bimodule (see, 
\cite[Proposition 1.10]{EKQR}). It is easy to check that 
\begin{equation*}
\{M(A_{\lambda });M(X_{\lambda });\pi _{\lambda \mu }^{M(A)};\chi _{\lambda
\mu }^{M(X)};\lambda ,\mu \in \Lambda ,\lambda \geq \mu \},
\end{equation*}%
where $\chi _{\lambda \mu }^{M(X)}=\chi _{\lambda \mu }^{L_{A}(A,X)}$ for
all $\lambda ,\mu \in \Lambda $ with $\lambda \geq \mu $, is an inverse
system of Hilbert $C^{\ast }$-bimodules. Then {{$\lim\limits_{\leftarrow
\lambda }$}} $M(X_{\lambda })$ has a structure of Hilbert {{$%
\lim\limits_{\leftarrow \lambda }$}} $M(A_{\lambda })-${{$%
\lim\limits_{\leftarrow \lambda }$}} $M(A_{\lambda })$ pro-$C^{\ast }$%
-bimodule. Moreover, by Lemma \ref{strict1} the maps $\chi _{\lambda \mu
}^{M(X)}:M(X_{\lambda })\rightarrow M(X_{\mu })$, $\lambda ,\mu \in \Lambda
,\lambda \geq \mu $ are all strictly continuous.

Consider, the maps: 
\begin{equation*}
\Phi :M(X)\rightarrow {{\lim\limits_{\leftarrow \lambda }}}M(X_{\lambda
}),\Phi \left( T\right) =\left( \chi _{\lambda }^{M(X)}(T)\right) _{\lambda }
\end{equation*}%
and 
\begin{equation*}
\varphi :M(A)\rightarrow {{\lim\limits_{\leftarrow \lambda }}}M(A_{\lambda
}),\varphi \left( m\right) =\left( \pi _{\lambda }^{M(A)}(m)\right)
_{\lambda }.
\end{equation*}%
It is easy to check that $\left( \Phi ,\varphi \right) $ is a morphism of
Hilbert pro-$C^{\ast }$-bimodules. Moreover, $\Phi $ is bijective, and since 
$\varphi $ is a pro-$C^{\ast }$-isomorphism, $\left( \Phi ,\varphi \right) $
is an isomorphism of Hilbert pro-$C^{\ast }$-bimodules. Clearly, a net $%
\{T_{i}\}_{i\in I}$ in $M(X)$ converges strictly to $0$ in $M(X)\ $if and
only if the net $\{\Phi \left( T_{i}\right) \}_{i\in I}$ converges strictly
to $0$ in {{$\lim\limits_{\leftarrow \lambda }$}}$M(X_{\lambda })$.
Therefore, the strict topology on $M(X)\ $can be identified with the inverse
limit of the strict topologies on $M(X_{\lambda }),$ $\lambda \in \Lambda ,$
and since $M(X_{\lambda }),$ $\lambda \in \Lambda $, are complete with
respect to the strict topology \cite[Proposition 1.27]{EKQR}, $M(X)$ is
complete with respect to the strict topology.

(2) Let $\lambda \in \Lambda $. By \cite{EKQR}, $\left( \iota _{X_{\lambda
}},\iota _{A_{\lambda }}\right) :\left( X_{\lambda },A_{\lambda }\right)
\rightarrow \left( M(X_{\lambda }),M(A_{\lambda })\right) $, where $\iota
_{X_{\lambda }}\left( \sigma _{\lambda }^{X}\left( x\right) \right) \left(
\pi _{\lambda }^{A}\left( a\right) \right) =\sigma _{\lambda }^{X}\left(
xa\right) $ and $\iota _{A_{\lambda }}\left( \pi _{\lambda }^{A}\left(
b\right) \right) \pi _{\lambda }^{A}\left( a\right) =\pi _{\lambda
}^{A}\left( ba\right) $ for all $x\in X$ and $a,b\in A$, is a morphism of
Hilbert $C^{\ast }$-bimodules. Since 
\begin{equation*}
\chi _{\lambda \mu }^{M(X)}\circ \iota _{X_{\lambda }}=\iota _{X_{\mu
}}\circ \sigma _{\lambda \mu }^{X}\text{ and }\pi _{\lambda \mu
}^{M(A)}\circ \iota _{A_{\lambda }}=\iota _{A_{\mu }}\circ \pi _{\lambda \mu
}^{A}
\end{equation*}%
for all $\lambda ,\mu \in \Lambda $ with $\lambda \geq \mu $, there is a
morphism of Hilbert pro-$C^{\ast }$-bimodules 
\begin{equation*}
\left( {{\lim\limits_{\leftarrow \lambda }}}\iota _{X_{\lambda }},{{%
\lim\limits_{\leftarrow \lambda }}}\iota _{A_{\lambda }}\right) :\left( {{%
\lim\limits_{\leftarrow \lambda }}}X_{\lambda },{{\lim\limits_{\leftarrow
\lambda }}}A_{\lambda }\right) \rightarrow \left( {{\lim\limits_{\leftarrow
\lambda }M(}}X_{\lambda }),{{\lim\limits_{\leftarrow \lambda }M(}}A_{\lambda
})\right) .
\end{equation*}%
Identifying $X$ with ${{\lim\limits_{\leftarrow \lambda }}}X_{\lambda }$ and 
$A$ with ${{\lim\limits_{\leftarrow \lambda }}}A_{\lambda }$, and using $%
\left( 1\right) $, we obtain a morphism of Hilbert pro-$C^{\ast }$-bimodules 
$\left( \iota _{X},\iota _{A}\right) :\left( X,A\right) \rightarrow \left(
M(X),M(A)\right) $, where $\iota _{X}\left( x\right) \left( a\right) =xa$
and $\iota _{A}\left( b\right) \left( a\right) =ba$ for all $x\in X$ and $%
a,b\in A$. We know that $\iota _{A}$ is nondegenerate and $XA$ is dense in $%
X $, and then $\left( \iota _{X},\iota _{A}\right) $ is nondegenerate.

(3) Since, for each $\lambda \in \Lambda ,$ 
\begin{equation*}
p_{\lambda }^{M(A)}(\iota _{X}(x))=\left\Vert \iota _{X_{\lambda }}\left(
\sigma _{\lambda }^{X}\left( x\right) \right) \right\Vert _{M(X_{\lambda
})}=\left\Vert \sigma _{\lambda }^{X}\left( x\right) \right\Vert
_{X_{\lambda }}=p_{\lambda }^{A}(x)
\end{equation*}%
for all $x\in X$, $X$ can be identified with a closed $M(A)-M(A)$ pro--$%
C^{\ast }$-sub-bimodule of $M(X)$. Using $(1)-\left( 2\right) $ and \cite[%
Chapter III, Theorem 3.1]{M}, we have 
\begin{eqnarray*}
\overline{\iota _{X}\left( X\right) }^{str} &=&{{\lim\limits_{\leftarrow
\lambda }}}\overline{{\chi }_{\lambda }^{M(X)}\left( \iota _{X}\left(
X\right) \right) }^{str}={{\lim\limits_{\leftarrow \lambda }}}\overline{%
\iota _{X_{\lambda }}\left( {\sigma }_{\lambda }^{X}\left( X\right) \right) }%
^{str}={{\lim\limits_{\leftarrow \lambda }}}\overline{\iota _{X_{\lambda
}}\left( X_{\lambda }\right) }^{str} \\
&=&{{\lim\limits_{\leftarrow \lambda }M(}}X_{\lambda })=M(X),
\end{eqnarray*}%
where $\overline{Z}^{str}$ denotes the closure with respect to the strict
topology of the Hilbert sub-bimodule $Z$ of a Hilbert bimodule $Y$.
Therefore, $X$ can be identified with a closed $M(A)-M(A)$ pro--$C^{\ast }$%
-sub-bimodule of $M(X)$ which is dense in $M(X)$ with respect to the strict
topology.
\end{proof}

\begin{remark}
Let $(X,A)$ be a full Hilbert pro-$C^{\ast }$-bimodule.

\begin{enumerate}
\item A net $\{x_{i}\}_{i\in I}$ in $X$ converges strictly to $0$ if and
only if the nets $\{p_{\lambda }^{A}\left( x_{i}a\right) \}_{i\in I}$ and $%
\{p_{\lambda }^{A}\left( ax_{i}\right) \}_{i\in I}$ converge to $0$ for all $%
a\in A$ and $\lambda \in \Lambda $.

\item The morphism of Hilbert pro-$C^{\ast }$-bimodules $\left( \iota
_{X},\iota _{A}\right) :\left( X,A\right) \rightarrow \left(
M(X),M(A)\right) $ is strictly continuous.
\end{enumerate}
\end{remark}

\begin{lemma}
\label{conv}Let $(X,A)$ be a full Hilbert pro-$C^{\ast }$-bimodule, let $%
\{e_{i}\}_{i\in I}$ be an approximate unit for $A$ and $T\in M(X).$ Then the
net $\{T\cdot e_{i}\}_{i\in I}$ converges strictly to $T$.
\end{lemma}

\begin{proof}
The net $\{T\cdot e_{i}\}_{i\in I}$ is bounded, since 
\begin{equation*}
p_{\lambda }^{M(A)}\left( T\cdot e_{i}\right) \leq p_{\lambda }^{M(A)}\left(
T\right) p_{\lambda ,L_{A}(A)}\left( e_{i}\right) =p_{\lambda }^{M(A)}\left(
T\right) p_{\lambda }\left( e_{i}\right) \leq p_{\lambda }^{M(A)}\left(
T\right)
\end{equation*}%
for all $i\in I$ and for $\lambda \in \Lambda .$ From 
\begin{equation*}
p_{\lambda }^{M(A)}\left( \left( T\cdot e_{i}-T\right) \left( a\right)
\right) =p_{\lambda }^{A}\left( T\left( e_{i}a-a\right) \right) \leq
p_{\lambda ,L_{A}\left( A,X\right) }\left( T\right) p_{\lambda }\left(
e_{i}a-a\right)
\end{equation*}%
for all $a\in A,$ $i\in I,$ $\lambda \in \Lambda $, and%
\begin{equation*}
p_{\lambda }\left( \left( \left( T\cdot e_{i}\right) ^{\ast }-T^{\ast
}\right) \left( x\right) \right) =p_{\lambda }\left( e_{i}T^{\ast
}(x)-T^{\ast }(x)\right)
\end{equation*}%
for all $x\in X,$ $i\in I,$ $\lambda \in \Lambda $, and Proposition \ref%
{bounded} , and taking into account that $\{e_{i}\}_{i\in I}$ is an
approximate unit for $A$, we conclude that $\{T\cdot e_{i}\}_{i\in I}\ $%
converges strictly to $T$.
\end{proof}

In the following theorem we show that any nondegenerate morphism of pro-$%
C^{\ast }$-bimodules is strictly continuous.

\begin{theorem}
\label{strict}Let $\left( X,A\right) $ and $\left( Y,B\right) $ be two full
Hilbert pro-$C^{\ast }$-bimodules and let $\left( \Phi ,\varphi \right) $ be
a nondegenerate \textit{morphism of Hilbert pro-}$C^{\ast }$\textit{%
-bimodules from }$\left( X,A\right) $ to $\left( M(Y),M(B)\right) $. Then $%
\left( \Phi ,\varphi \right) $ extends to a unique nondegenerate morphism of 
\textit{Hilbert pro-}$C^{\ast }$\textit{-bimodules }$\left( \overline{\Phi },%
\overline{\varphi }\right) $ from $\left( M(X),M(A)\right) $ to $\left(
M(Y),M(B)\right) $. Moreover, $\overline{\Phi }$ is strictly continuous.
\end{theorem}

\begin{proof}
For each $\delta \in \Delta $, there is $\lambda \in \Lambda $ such that $%
q_{\delta ,M(B)}\left( \varphi \left( a\right) \right) \leq p_{\lambda
}\left( a\right) $ for all $a\in A$ and $q_{\delta }^{M(B)}\left( \Phi
\left( x\right) \right) \leq p_{\lambda }^{A}\left( x\right) $ for all $x\in
X$. So there exists a $C^{\ast }$-morphism $\varphi _{\left( \lambda ,\delta
\right) }:A_{\lambda }\rightarrow M(B_{\delta })$ such that $\varphi
_{\left( \lambda ,\delta \right) }\circ \pi _{\lambda }^{A}=\pi _{\delta
}^{M(B)}\circ \varphi $ and a linear map $\Phi _{\left( \lambda ,\delta
\right) }:X_{\lambda }\rightarrow M(Y_{\delta })$ such that $\Phi _{\left(
\lambda ,\delta \right) }\circ \sigma _{\lambda }^{X}=\chi _{\delta
}^{M(Y)}\circ \Phi $. It is easy to check that $\left( \Phi _{\left( \lambda
,\delta \right) },\varphi _{\left( \lambda ,\delta \right) }\right) $ is a
morphism of Hilbert $C^{\ast }$-bimodules from $\left( X_{\lambda
},A_{\lambda }\right) \ $to $\left( M(Y_{\delta }),M(B_{\delta })\right) $.
Moreover, $\left( \Phi _{\left( \lambda ,\delta \right) },\varphi _{\left(
\lambda ,\delta \right) }\right) $ is nondegenerate, since 
\begin{equation*}
\left[ \varphi _{\left( \lambda ,\delta \right) }\left( A_{\lambda }\right)
B_{\delta }\right] =\left[ \varphi _{\left( \lambda ,\delta \right) }\left(
\pi _{\lambda }^{A}\left( A\right) \right) B_{\delta }\right] =\left[ \pi
_{\delta }^{M(B)}\left( \varphi \left( A\right) B\right) \right] =\left[ \pi
_{\delta }^{M(B)}\left( B\right) \right] =B_{\delta }
\end{equation*}%
and 
\begin{eqnarray*}
\left[ \Phi _{\left( \lambda ,\delta \right) }\left( X_{\lambda }\right)
B_{\delta }\right] &=&\left[ \Phi _{\left( \lambda ,\delta \right) }\left(
\sigma _{\lambda }^{X}\left( X\right) \right) \pi _{\delta }^{M(B)}\left(
B\right) \right] =\left[ \chi _{\delta }^{M(Y)}\left( \Phi \left( X\right)
B\right) \right] \\
&=&\left[ \sigma _{\delta }^{Y}\left( Y\right) \right] =Y_{\delta }.
\end{eqnarray*}

Then, by \cite[Theorem 1.30]{EKQR}, $\Phi _{\left( \lambda ,\delta \right) }$
is strictly continuous and $\left( \Phi _{\left( \lambda ,\delta \right)
},\varphi _{\left( \lambda ,\delta \right) }\right) $ extends to a unique
nondegenerate morphism of Hilbert $C^{\ast }$-modules $\left( \overline{\Phi
_{\left( \lambda ,\delta \right) }},\overline{\varphi _{\left( \lambda
,\delta \right) }}\right) $ \ from $\left( M(X_{\lambda }),M(A_{\lambda
})\right) \ $to $\left( M(Y_{\delta }),M(B_{\delta })\right) $. Let $%
\overline{\Phi _{\delta }}=\overline{\Phi _{\left( \lambda ,\delta \right) }}%
\circ \chi _{\lambda }^{M(X)}$ and $\overline{\varphi _{\delta }}=\overline{%
\varphi _{\left( \lambda ,\delta \right) }}\circ \pi _{\lambda }^{M(A)}$.
Clearly, $\left( \overline{\Phi _{\delta }},\overline{\varphi _{\delta }}%
\right) $ is a morphism of pro-$C^{\ast }$-bimodules from $\left(
M(X),M(A)\right) $ to $\left( M(Y_{\delta }),M(B_{\delta })\right) $.
Moreover, $\overline{\Phi _{\delta }}$ is strictly continuous, since $\chi
_{\lambda }^{M(X)}$ is strictly continuos (see Lemma \ref{strict1}).

Let $\delta _{1},\delta _{2}\in \Delta \ $with $\delta _{1}\geq \delta _{2}$%
. We have 
\begin{eqnarray*}
\overline{\Phi _{\delta _{1}}}\left( \iota _{X}(x)\right) &=&\left( 
\overline{\Phi _{\left( \lambda _{1},\delta _{1}\right) }}\circ \chi
_{\lambda }^{M(X)}\right) \left( \iota _{X}(x)\right) =\overline{\Phi
_{\left( \lambda _{1},\delta _{1}\right) }}\left( \iota _{X_{\lambda
_{1}}}(\sigma _{\lambda _{1}}^{X}\left( x\right) )\right) \\
&=&\Phi _{\left( \lambda _{1},\delta _{1}\right) }(\sigma _{\lambda
_{1}}^{X}\left( x\right) )=\chi _{\delta _{1}}^{M(Y)}\left( \Phi \left(
x\right) \right)
\end{eqnarray*}%
for some $\lambda _{1}\in \Lambda $ and for all $x\in X$. Then 
\begin{equation*}
\left( \chi _{\delta _{1}\delta _{2}}^{M(Y)}\circ \overline{\Phi _{\delta
_{1}}}\right) \left( \iota _{X}(x)\right) =\chi _{\delta _{1}\delta
_{2}}^{M(Y)}\left( \chi _{\delta _{1}}^{M(Y)}\left( \Phi \left( x\right)
\right) \right) =\chi _{\delta _{2}}^{M(Y)}\left( \Phi \left( x\right)
\right) =\overline{\Phi _{\delta _{2}}}\left( \iota _{X}(x)\right)
\end{equation*}%
for all $x\in X$. From these relations and taking into account that $\chi
_{\delta _{1}\delta _{2}}^{M(Y)},$ $\overline{\Phi _{\delta _{1}}},$ $%
\overline{\Phi _{\delta _{2}}}$ are strictly continuous and $X$ is dense in $%
M(X)\ $with respect to the strict topology, we conclude that $\chi _{\delta
_{1}\delta _{2}}^{M(Y)}\circ \overline{\Phi _{\delta _{1}}}=\overline{\Phi
_{\delta _{2}}}$. Therefore there is a strictly continuous linear map $%
\overline{\Phi }:M(X)\rightarrow M(Y)\ $such that $\chi _{\delta
}^{M(Y)}\circ \overline{\Phi }=\overline{\Phi _{\delta }}$ for all $\delta
\in \Delta $, and $\overline{\Phi }\circ \iota _{X}=\Phi $.

By \cite[Proposition 3.15]{P}, there is a pro-$C^{\ast }$-morphism $%
\overline{\varphi }:M(A)\rightarrow M(B)$ such that $\pi _{\delta
}^{M(B)}\circ \overline{\varphi }=\overline{\varphi _{\delta }}$ for all $%
\delta \in \Delta $ and $\overline{\varphi }\circ \iota _{A}=\varphi .$

It is easy to check that $\left( \overline{\Phi },\overline{\varphi }\right) 
$ is a morphism of Hilbert pro-$C^{\ast }$-bimodules. Since $\overline{%
\varphi }$ is nondegenerate \cite[Proposition 6.1.4]{J1} and 
\begin{eqnarray*}
\left[ \overline{\Phi }\left( M(X)\right) B\right] &=&\left[ \overline{\Phi }%
\left( M(X)\right) \varphi \left( A\right) B\right] =\left[ \overline{\Phi }%
\left( M(X)A)\right) B\right] \\
&=&\left[ \Phi \left( X)\right) B\right] =Y
\end{eqnarray*}%
the morphism of Hilbert pro-$C^{\ast }$-bimodule $\left( \overline{\Phi }\ ,%
\overline{\varphi }\right) \ $is nondegenerate.

Suppose that there is another morphism of Hilbert pro-$C^{\ast }$-bimodules $%
\left( \Phi _{1},\varphi _{1}\right) :\left( M(X),M(A)\right) \rightarrow
\left( M(Y),M(B)\right) $ such that $\Phi _{1}\left( \iota _{X}(x)\right)
=\Phi \left( x\right) $ for all $x\in X$ and $\varphi _{1}\left( \iota
_{A}\left( a\right) \right) =\varphi \left( a\right) $ for all $a\in A$. Let 
$\{e_{i}\}_{i\in I}$ be a approximate unit for $A$. Then, by Lemma \ref{conv}
for each $T\in M(X)$ and $m\in M(A)$, the nets $\{T\cdot e_{i}\}_{i\in I}$
and $\{m\cdot e_{i}\}_{i\in I}$ are strictly convergent to $T$ respectively $%
m$. Thus we have 
\begin{equation*}
\Phi _{1}\left( T\right) =\text{str-}\lim\limits_{i}\Phi _{1}\left( T\cdot
e_{i}\right) =\text{str-}\lim\limits_{i}\Phi \left( T\cdot e_{i}\right) =%
\overline{\Phi }\left( T\right)
\end{equation*}%
for all $T\in M(X)$ and 
\begin{equation*}
\varphi _{1}\left( m\right) =\text{str-}\lim\limits_{i}\varphi _{1}\left(
m\cdot e_{i}\right) =\text{str-}\lim\limits_{i}\varphi \left( m\cdot
e_{i}\right) =\overline{\varphi }\left( m\right)
\end{equation*}%
for all $m\in M(A)$.
\end{proof}

Let $X$ be a Hilbert $A-A$ pro-$C^{\ast }$-bimodule. For a closed two sided
ideal $\mathcal{I}$ of $A$ we put $\mathcal{I}X=$span$\{ax/a\in \mathcal{I}%
,x\in X\}$ and $X\mathcal{I}=$span$\{xa/a\in \mathcal{I},x\in X\}$. By \cite[%
Lemma 3.7]{JZ2}, $\mathcal{I}X$ and $X\mathcal{I}$ are closed Hilbert pro-$%
C^{\ast }$-sub-bimodules of $X$.

\begin{definition}
Let $(X,A)$ and \thinspace $(Y,C)$ be two Hilbert pro-$C^{\ast }$-bimodules.
We say that $(Y,C)$ is an extension of $(X,A)$ if the following conditions
are satisfied:

\begin{enumerate}
\item $C$ contains $A$ as an ideal;

\item there exists a morphism $\left( \varphi _{X},\varphi _{A}\right) $ of
Hilbert pro-$C^{\ast }$-bimodules from $(X,A)$ to $\,(Y,C),$ such that $%
\varphi _{A}:A\rightarrow C$ is just the inclusion map;

\item $\varphi _{X}\left( X\right) =\varphi _{A}\left( A\right) Y=Y\varphi
_{A}\left( A\right) $.
\end{enumerate}
\end{definition}

\begin{remark}
\label{extension}If $(Y,C)$ is an extension of $(X,A)$, and if the topology
on $C$ is given by the family of $C^{\ast }$-seminorms $\{p_{\lambda
};\lambda \in \Lambda \}$, then the topology on $A$ is given by $%
\{p_{\lambda }|_{A};\lambda \in \Lambda \}$, and $p_{\lambda }\left( \varphi
_{A}\left( a\right) \right) =p_{\lambda }\left( a\right) $ for all $a\in A$
and for all $\lambda \in \Lambda .$ Therefore, $p_{\lambda }^{C}\left(
\varphi _{X}\left( x\right) \right) =p_{\lambda }^{A}\left( x\right) $ for
all $x\in X$ and for all $\lambda \in \Lambda ,$ and so, for each $\lambda
\in \Lambda $, there is a linear map $\varphi _{X_{\lambda }}:X_{\lambda
}\rightarrow Y_{\lambda }$ such that $\sigma _{\lambda }^{Y}\circ \varphi
_{X}=\varphi _{X_{\lambda }}\circ \sigma _{\lambda }^{X}$. Then $\varphi
_{X}=\lim\limits_{\leftarrow \lambda }\varphi _{X_{\lambda }}$, and for each 
$\lambda \in \Lambda ,$ $(Y_{\lambda },C_{\lambda })$ is an extension of $%
(X_{\lambda },A_{\lambda })$ via the morphism $\left( \varphi _{X_{\lambda
}},\varphi _{A_{\lambda }}\right) $, where $\varphi _{A_{\lambda }}$ is the
inclusion of $A_{\lambda }$ into $C_{\lambda }$.
\end{remark}

In the following proposition, we show that $\left( M(X),M(A)\right) $ is a
maximal extension of $(X,A)$ in the sense that if $(Y,C)$ is another
extension of $(X,A)$ via a morphism $(\psi _{X},\psi _{A}),$ then there is a
morphism of Hilbert pro-$C^{\ast }$-bimodules $(\vartheta _{Y},\vartheta
_{C}):(Y,C)\rightarrow (M(X),M(A))$ such that $\vartheta _{Y}\circ \psi
_{X}=\iota _{X}$ and $\vartheta _{C}\circ \psi _{A}=\iota _{A}$ ( see, for
the case of Hilbert $C^{\ast }$-modules, \cite{BG1,BG2}).

\begin{proposition}
Let $X$ be a full Hilbert pro-$C^{\ast }$-bimodule over $A$. Then $\left(
M(X),M(A)\right) $ is a maximal extension of $(X,A)$.
\end{proposition}

\begin{proof}
Let $(\iota _{X},\,\iota _{A})$ be the morphism of Theorem \ref{multiplier}%
(2) between $(X,A)$ and $(M(X),M(A))$, where $\iota _{X}(x)(a)=xa,\,\,\iota
_{A}(a)(b)=ab,$ for $x\in X,\,a,b\,\in A$. From \cite[Corollary 3.3]{R} we
have that for every $\lambda \in \Lambda ,$ $M(X_{\lambda })\iota
_{A_{\lambda }}(A_{\lambda })=\iota _{X_{\lambda }}(X_{\lambda })=\iota
_{A_{\lambda }}(A_{\lambda })M(X_{\lambda })$. Therefore, since from Theorem %
\ref{multiplier}, we have that $M(X)={{\lim\limits_{\leftarrow \lambda }}}%
M(X_{\lambda }),\,\iota _{X}={{\lim\limits_{\leftarrow \lambda }}}\iota
(X_{\lambda }),\,\iota _{A}={{\lim\limits_{\leftarrow \lambda }}}\iota
(A_{\lambda }),$ and since both $\iota _{A}(A)M(X),M(X)\iota _{A}(A)$ and $%
\iota _{X}(X)$ are closed submodules of $M(X)$, we deduce that $\iota
_{A}(A)M(X)=\iota _{X}(X)=M(X)\iota _{A}(A)$. Hence $(M(X),M(A))$ is an
extension of $(X,A)$.

To show that $(M(X),M(A))$ is a maximal extension, let $(Y,C)$ be another
extension of $(X,A)$ via a morphism $(\psi _{X},\psi _{A})$. Then, by Remark %
\ref{extension}, $\psi _{X}=\lim\limits_{\leftarrow \lambda }\psi
_{X_{\lambda }},\psi _{A}=\lim\limits_{\leftarrow \lambda }\psi _{A_{\lambda
}},$ and for each $\lambda \in \Lambda ,$ $(Y_{\lambda },C_{\lambda })$ is
an extension of $(X_{\lambda },A_{\lambda })$ via the morphism $\left( \psi
_{X_{\lambda }},\psi _{A_{\lambda }}\right) $. By \cite[Proposition 3.4]{R},
there exists a unique morphism $(\vartheta _{Y_{\lambda }},\vartheta
_{C_{\lambda }}):(Y_{\lambda },C_{\lambda })\rightarrow (M(X_{\lambda
}),M(A_{\lambda }))$ such that $\vartheta _{Y_{\lambda }}\circ \psi
_{X_{\lambda }}=\iota _{X_{\lambda }}$ and $\vartheta _{C_{\lambda }}\circ
\psi _{A_{\lambda }}=\,\iota _{A_{\lambda }}$. Moreover, $\ \ $%
\begin{equation*}
\vartheta _{Y_{\lambda }}\left( \sigma _{\lambda }^{Y}\left( y\right)
\right) \left( \pi _{\lambda }^{A}\left( a\right) \right) =\psi _{X_{\lambda
}}^{-1}\left( \sigma _{\lambda }^{Y}\left( y\right) \psi _{A_{\lambda
}}\left( \pi _{\lambda }^{A}\left( a\right) \right) \right)
\end{equation*}%
and\ \ \ \ \ \ 
\begin{equation*}
\vartheta _{C_{\lambda }}\left( \pi _{\lambda }^{C}\left( c\right) \right)
\left( \pi _{\lambda }^{A}\left( a\right) \right) =\psi _{A_{\lambda
}}^{-1}\left( \pi _{\lambda }^{C}\left( c\right) \psi _{A_{\lambda }}\left(
\pi _{\lambda }^{A}\left( a\right) \right) \right)
\end{equation*}%
for all $a\in A$, for all $c\in C$ and for all $y\in Y$. It is easy to check
that $\left( \vartheta _{Y_{\lambda }}\right) _{\lambda }$ is an inverse
system of linear maps, $\left( \vartheta _{C_{\lambda }}\right) _{\lambda }$
is an inverse system of $C^{\ast }$-morphisms, and $(\vartheta
_{Y},\vartheta _{C}):(Y,C)\rightarrow (M(X),M(A))$, where $\vartheta
_{Y}=\lim\limits_{\leftarrow \lambda }\vartheta _{Y_{\lambda }}$ and $%
\vartheta _{C}=\lim\limits_{\leftarrow \lambda }\vartheta _{C_{\lambda }}$,
is a morphism of Hilbert pro-$C^{\ast }$-bimodules such that $\vartheta
_{Y}\circ \psi _{X}=\iota _{X}$ and $\vartheta _{C}\circ \psi _{A}=\,\iota
_{A}$.
\end{proof}

\section{Crossed products by Hilbert pro-$C^{\ast }$-modules}

\textit{A covariant representation }of a Hilbert pro-$C^{\ast }$-bimodule $%
\left( X,A\right) $ on a pro-$C^{\ast }$-algebra $B$ is a morphism of
Hilbert pro-$C^{\ast }$-bimodules from $\left( X,A\right) $ to the Hilbert
pro-$C^{\ast }$-bimodule $\left( B,B\right) $.

\textit{The crossed product of }$A$\textit{\ by }a Hilbert pro-$C^{\ast }$%
-bimodule $\left( X,A\right) $\textit{\ }is a pro-$C^{\ast }$-algebra,
denoted by $A\times _{X}\mathbb{Z}$, and a covariant representation $\left(
i_{X},i_{A}\right) $ of $\left( X,A\right) $ on $A\times _{X}\mathbb{Z}$
with the property that for any covariant representation $\left( \varphi
_{X},\varphi _{A}\right) $ of $\left( X,A\right) $ on a pro-$C^{\ast }$%
-algebra $B$, there is a unique pro-$C^{\ast }$-morphism $\Phi :A\times _{X}%
\mathbb{Z}\rightarrow B$ such that $\Phi \circ i_{X}=\varphi _{X}$ and $\Phi
\circ i_{A}=\varphi _{A}$ \cite[Definition 3.3]{JZ}.

\begin{remark}
If $\left( \Phi ,\varphi \right) $ is a morphism of Hilbert pro-$C^{\ast }$%
-bimodules from $\left( X,A\right) $ to $\left( Y,B\right) $, then $\left(
i_{Y}\circ \Phi ,i_{B}\circ \varphi \right) $ is a covariant representation
of $X$ on $B\times _{Y}\mathbb{Z}$ and by the universal property of $A\times
_{X}\mathbb{Z}$ there is a unique pro-$C^{\ast }$-morphism $\Phi \times
\varphi $ from $A\times _{X}\mathbb{Z}$ to $B\times _{Y}\mathbb{Z}$ such
that $\left( \Phi \times \varphi \right) \circ i_{A}=i_{B}\circ \varphi $
and $\left( \Phi \times \varphi \right) \circ i_{X}=i_{Y}\circ \Phi $.
\end{remark}

\begin{lemma}
\label{Conjugate} Let {$\left( \Phi ,\varphi \right) $ be a } a morphism of
Hilbert pro-$C^{\ast }$-bimodules from $\left( X,A\right) $ to $\left(
Y,B\right) $. If $\Gamma $ and $\Gamma ^{\prime }$ have the same index set
and $\varphi =\lim\limits_{\leftarrow \lambda }\varphi _{\lambda }$, then $%
\Phi =\lim\limits_{\leftarrow \lambda }\Phi _{\lambda },$ for each $\lambda
\in \Lambda ,\left( \Phi _{\lambda },\varphi _{\lambda }\right) $ is a
morphism of Hilbert $C^{\ast }$-bimodules, $\left( \Phi _{\lambda }\times
\varphi _{\lambda }\right) _{\lambda }$ is an inverse system of $C^{\ast }$%
-morphisms and $\Phi \times \varphi =\lim\limits_{\leftarrow \lambda }\Phi
_{\lambda }\times \varphi _{\lambda }$. Moreover, if {$\left( \Phi ,\varphi
\right) $ is an isomorphism of }Hilbert pro-$C^{\ast }$-bimodules and $%
\varphi _{\lambda },\lambda \in \Lambda $ are $C^{\ast }$-isomorphisms, then 
$\left( \Phi _{\lambda },\varphi _{\lambda }\right) ,$ $\lambda \in \Lambda $
are isomorphisms of Hilbert $C^{\ast }$-bimodules.
\end{lemma}

\begin{proof}
Let $\lambda \in \Lambda $. From 
\begin{equation*}
q_{\lambda }^{B}\left( \Phi \left( x\right) \right) ^{2}=q_{\lambda }\left(
\varphi \left( \left\langle x,x\right\rangle \right) \right) \leq p_{\lambda
}\left( \left\langle x,x\right\rangle \right) =p_{\lambda }^{A}\left(
x\right) ^{2}
\end{equation*}%
for all $x\in X$, we deduce that there is a linear map $\Phi _{\lambda
}:X_{\lambda }\rightarrow Y_{\lambda }$ such that $\Phi _{\lambda }\circ
\sigma _{\lambda }^{X}=\sigma _{\lambda }^{Y}\circ \Phi $. It is easy to
verify that $\left( \Phi _{\lambda }\right) _{\lambda }$ is an inverse
system of linear maps and $\Phi =\lim\limits_{\leftarrow \lambda }\Phi
_{\lambda }$. Moreover, for each $\lambda \in \Lambda $, $\left( \Phi
_{\lambda },\varphi _{\lambda }\right) $ is a morphism of Hilbert $C^{\ast }$%
-bimodules. Let $\Phi _{\lambda }\times \varphi _{\lambda }$ be the $C^{\ast
}$-morphism from $A_{\lambda }\times _{X_{\lambda }}\mathbb{Z}$ to $%
B_{\lambda }\times _{Y_{\lambda }}\mathbb{Z}$ induced by $\left( \Phi
_{\lambda },\varphi _{\lambda }\right) $. From 
\begin{eqnarray*}
\pi _{\lambda \mu }^{B\times _{Y}\mathbb{Z}}\circ \left( \Phi _{\lambda
}\times \varphi _{\lambda }\right) \circ i_{A_{\lambda }} &=&\pi _{\lambda
\mu }^{B\times _{Y}\mathbb{Z}}\circ i_{B_{\lambda }}\circ \varphi _{\lambda
}=i_{B_{\mu }}\circ \pi _{\lambda \mu }^{B}\circ \varphi _{\lambda } \\
&=&i_{B_{\mu }}\circ \varphi _{\mu }\circ \pi _{\lambda \mu }^{A}=\left(
\Phi _{\mu }\times \varphi _{\mu }\right) \circ \pi _{\lambda \mu }^{A\times
_{X}\mathbb{Z}}\circ i_{A_{\lambda }}
\end{eqnarray*}%
and 
\begin{equation*}
\pi _{\lambda \mu }^{B\times _{Y}\mathbb{Z}}\circ \left( \Phi _{\lambda
}\times \varphi _{\lambda }\right) \circ i_{X_{\lambda }}=\left( \Phi _{\mu
}\times \varphi _{\mu }\right) \circ \pi _{\lambda \mu }^{A\times _{X}%
\mathbb{Z}}\circ i_{X_{\lambda }}
\end{equation*}%
for all $\lambda ,\mu \in \Lambda $ with $\lambda \geq \mu $ and taking into
account $i_{A_{\lambda }}\left( A_{\lambda }\right) $ and $i_{X_{\lambda
}}\left( X_{\lambda }\right) $ generate $A_{\lambda }\times _{X_{\lambda }}%
\mathbb{Z}$, we deduce that $\left( \Phi _{\lambda }\times \varphi _{\lambda
}\right) _{\lambda }$ is an inverse system of $C^{\ast }$-morphisms.
Moreover, since 
\begin{equation*}
\lim\limits_{\leftarrow \lambda }\left( \Phi _{\lambda }\times \varphi
_{\lambda }\right) \circ \lim\limits_{\leftarrow \lambda }i_{A_{\lambda
}}=\lim\limits_{\leftarrow \lambda }\left( \Phi _{\lambda }\times \varphi
_{\lambda }\right) \circ i_{A_{\lambda }}=\lim\limits_{\leftarrow \lambda
}i_{B_{\lambda }}\circ \varphi _{\lambda }=\lim\limits_{\leftarrow \lambda
}i_{B_{\lambda }}\circ \lim\limits_{\leftarrow \lambda }\varphi _{\lambda }
\end{equation*}%
and 
\begin{equation*}
\lim\limits_{\leftarrow \lambda }\left( \Phi _{\lambda }\times \varphi
_{\lambda }\right) \circ \lim\limits_{\leftarrow \lambda }i_{X_{\lambda
}}=\lim\limits_{\leftarrow \lambda }\left( \Phi _{\lambda }\times \varphi
_{\lambda }\right) \circ i_{X_{\lambda }}=\lim\limits_{\leftarrow \lambda
}i_{Y_{\lambda }}\circ \Phi _{\lambda }=\lim\limits_{\leftarrow \lambda
}i_{Y_{\lambda }}\circ \lim\limits_{\leftarrow \lambda }\Phi _{\lambda },
\end{equation*}%
we obtain $\Phi \times \varphi =\lim\limits_{\leftarrow \lambda }\Phi
_{\lambda }\times \varphi _{\lambda }$.

Suppose that {$\left( \Phi ,\varphi \right) $ is an isomorphism of }Hilbert
pro-$C^{\ast }$-bimodules and $\varphi _{\lambda },\lambda \in \Lambda $ are 
$C^{\ast }$-isomorphisms. Then, since $\varphi ^{-1}=\lim\limits_{\leftarrow
\lambda }\varphi _{\lambda }^{-1}$, by the first part of the proof, $\Phi
^{-1}=\lim\limits_{\leftarrow \lambda }\psi _{\lambda }\ $and $\left( \psi
_{\lambda },\varphi _{\lambda }^{-1}\right) $ is a morphism of Hilbert $%
C^{\ast }$-bimodules for all $\lambda \in \Lambda $. Let $\lambda \in
\Lambda $. From%
\begin{equation*}
\psi _{\lambda }\circ \Phi _{\lambda }\circ \sigma _{\lambda }^{X}=\psi
_{\lambda }\circ \sigma _{\lambda }^{Y}\circ \Phi =\sigma _{\lambda
}^{X}\circ \Phi ^{-1}\circ \Phi =\sigma _{\lambda }^{X}
\end{equation*}%
and 
\begin{equation*}
\Phi _{\lambda }\circ \psi _{\lambda }\circ \sigma _{\lambda }^{Y}=\Phi
_{\lambda }\circ \sigma _{\lambda }^{X}\circ \Phi ^{-1}=\sigma _{\lambda
}^{Y}\circ \Phi \circ \Phi ^{-1}=\sigma _{\lambda }^{Y}
\end{equation*}%
and taking into account that $\sigma _{\lambda }^{X}$ and $\sigma _{\lambda
}^{Y}$ are surjective, we deduce that $\psi _{\lambda }=\Phi _{\lambda
}^{-1} $.
\end{proof}

The following proposition gives the relation between the crossed product of $%
A$ by $X$ and the crossed product of $M(A)$ by $M(X)$.

\begin{proposition}
Let $\left( X,A\right) $ be full Hilbert pro-$C^{\ast }$-bimodule. Then $%
A\times _{X}\mathbb{Z}$ can be embedded into $M\mathbb{(}A)\times _{M(X)}%
\mathbb{Z}$.
\end{proposition}

\begin{proof}
Let $\iota _{A}$ be the embedding of $A$ in $M(A)$ and $\iota _{X}$ the
embedding of $X$ in $M(X)$. Then $\left( \iota _{X},\iota _{A}\right) $ is a
morphism of Hilbert pro-$C^{\ast }$-bimodules, and since $\iota
_{A}=\lim\limits_{\leftarrow \lambda }\iota _{A_{\lambda }}$, by Lemma \ref%
{Conjugate}, $\iota _{X}\times \iota _{A}=$ $\lim\limits_{\leftarrow \lambda
}\iota _{X_{\lambda }}\times \iota _{A_{\lambda }}$ is a pro-$C^{\ast }$%
-morphism from $A\times _{X}\mathbb{Z}$ to $M\mathbb{(}A)\times _{M(X)}%
\mathbb{Z}$. Moreover, since 
\begin{eqnarray*}
p_{\lambda ,M\mathbb{(}A)\times _{M(X)}\mathbb{Z}}\left( \iota _{X}\times
\iota _{A}\left( c\right) \right) &=&\left\Vert \iota _{X_{\lambda }}\times
\iota _{A_{\lambda }}\left( \pi _{\lambda }^{A\times _{X}\mathbb{Z}}\left(
c\right) \right) \right\Vert _{M\mathbb{(}A_{\lambda })\times _{M(X_{\lambda
})}\mathbb{Z}} \\
&&\text{\cite[Remark 2.2]{A}} \\
&=&\left\Vert \pi _{\lambda }^{A\times _{X}\mathbb{Z}}\left( c\right)
\right\Vert _{A_{\lambda }\times _{X_{\lambda }}\mathbb{Z}}=p_{\lambda
,A\times _{X}\mathbb{Z}}\left( c\right)
\end{eqnarray*}%
for all $c\in A\times _{X}\mathbb{Z}$ and for all $\lambda \in \Lambda $, $%
A\times _{X}\mathbb{Z}$ can be identified with a pro-$C^{\ast }$-subalgebra
of $M\mathbb{(}A)\times _{M(X)}\mathbb{Z}$.
\end{proof}

The following proposition is a generalization of \cite[Proposition 4.7]{R}.

\begin{proposition}
Let $\left( X,A\right) $ be a full Hilbert pro-$C^{\ast }$-bimodule. Then $%
M(A)\times _{M(X)}\mathbb{Z}$ can be identified with a pro-$C^{\ast }$%
-subalgebra of $M(A\times _{X}\mathbb{Z})$.
\end{proposition}

\begin{proof}
Since $X$ is full, $\left( i_{X},i_{A}\right) $ is nondegenerate and $%
i_{A}=\lim\limits_{\leftarrow \lambda }i_{A_{\lambda }}$ and $i_{X}=$ $%
\lim\limits_{\leftarrow \lambda }i_{X_{\lambda }}$ \cite[Propositions 3.4
and 3.5]{JZ}. Then, by Theorem \ref{strict}, $\left( i_{X},i_{A}\right) $
extends to a covariant representation $\left( \overline{i_{X}},\overline{%
i_{A}}\right) $ of $\left( M(X),M(A)\right) $ on $M(A\times _{X}\mathbb{Z})$%
, and moreover, $\overline{i_{A}}=\lim\limits_{\leftarrow \lambda }\overline{%
i_{A_{\lambda }}}$ \ and $\overline{i_{X}}=$ $\lim\limits_{\leftarrow
\lambda }\overline{i_{X_{\lambda }}}$. It is easy to check $\left( \overline{%
i_{X_{\lambda }}},\overline{i_{A_{\lambda }}}\right) $ is a covariant
representation of $\left( M(X_{\lambda }),M(A_{\lambda })\right) $ on $%
M(A_{\lambda }\times _{X_{\lambda }}\mathbb{Z})$ for each $\lambda \in
\Lambda $. By \cite[Proposition 4.7]{R}, for each $\lambda \in \Lambda $,
there is an injective $C^{\ast }$-morphism $\Phi _{\lambda }:M(A_{\lambda
})\times _{M(X_{\lambda })}\mathbb{Z}$ $\rightarrow M(A_{\lambda }\times
_{X_{\lambda }}\mathbb{Z})$ such that $\Phi _{\lambda }\circ i_{M(X_{\lambda
})}=\overline{i_{X_{\lambda }}}$ and $\Phi _{\lambda }\circ i_{M(A_{\lambda
})}=\overline{i_{A_{\lambda }}}\ $. From%
\begin{eqnarray*}
\pi _{\lambda \mu }^{M(A\times _{X}\mathbb{Z})}\circ \Phi _{\lambda }\circ
i_{M(X_{\lambda })} &=&\pi _{\lambda \mu }^{M(A\times _{X}\mathbb{Z})}\circ 
\overline{i_{X_{\lambda }}}=\overline{i_{X_{\mu }}}\circ \chi _{\lambda \mu
}^{M(X)} \\
&=&\Phi _{\mu }\circ i_{M(X_{\mu })}\circ \chi _{\lambda \mu }^{M(X)}=\Phi
_{\mu }\circ \pi _{\lambda \mu }^{M(A\times _{X}\mathbb{Z})}\circ
i_{M(X_{\lambda })}
\end{eqnarray*}%
and 
\begin{eqnarray*}
\pi _{\lambda \mu }^{M(A\times _{X}\mathbb{Z})}\circ \Phi _{\lambda }\circ
i_{M(A_{\lambda })} &=&\pi _{\lambda \mu }^{M(A\times _{X}\mathbb{Z})}\circ 
\overline{i_{A_{\lambda }}}=\overline{i_{A_{\mu }}}\circ \pi _{\lambda \mu
}^{M(A)} \\
&=&\Phi _{\mu }\circ i_{M(A_{\mu })}\circ \pi _{\lambda \mu }^{M(A)}=\Phi
_{\mu }\circ \pi _{\lambda \mu }^{M(A\times _{X}\mathbb{Z})}\circ
i_{M(A_{\lambda })}
\end{eqnarray*}%
for all $\lambda ,\mu \in \Lambda $, with $\lambda \geq \mu $, and taking
into account that $i_{M(X_{\lambda })}\left( M(X_{\lambda })\right) $ and $%
i_{M(A_{\lambda })}\left( M(A_{\lambda })\right) $ generate\ $M(A_{\lambda
}) $ $\times _{M(X_{\lambda })}\mathbb{Z}$, we deduce that $\left( \Phi
_{\lambda }\right) _{\lambda }$ is an inverse system of isometric $C^{\ast }$%
-morphisms, and then $\Phi =$ $\lim\limits_{\leftarrow \lambda }\Phi
_{\lambda }$ is an injective pro-$C^{\ast }$-morphism from $%
\lim\limits_{\leftarrow \lambda }M(A_{\lambda })$\ $\times _{M(X_{\lambda })}%
\mathbb{Z}$ to $\lim\limits_{\leftarrow \lambda }M(A_{\lambda }\times
_{X_{\lambda }}\mathbb{Z})$ such that $p_{\lambda ,M(A\times _{X}\mathbb{Z}%
)}\left( \Phi \left( c\right) \right) =p_{\lambda ,M(A)\times _{M(X)}\mathbb{%
Z}}\left( c\right) $ for all $c\in M(A)\times _{M(X)}\mathbb{Z}$ and for all 
$\lambda \in \Lambda $. Therefore, $M(A)\times _{M(X)}\mathbb{Z}$ can be
identified with a pro-$C^{\ast }$-subalgebra of $M(A\times _{X}\mathbb{Z})$.
\end{proof}

An automorphism $\alpha $ of a pro-$C^{\ast }$-algebra $A$ such that $%
p_{\lambda }(\alpha (a))=p_{\lambda }(a)$ for all $a\in A$ and $\lambda \in
\Lambda ^{^{\prime }}$, where $\Lambda ^{^{\prime }}$ is a cofinal subset of 
$\Lambda $, is called an inverse limit automorphism. If $\alpha $ is an
inverse limit automorphism of the pro-$C^{\ast }$-algebra $A$, then $%
X_{\alpha }=\{\xi _{x};x\in A\}$ is a Hilbert $A-A$ pro-$C^{\ast }$-bimodule
with the bimodule structure defined as $\xi _{x}a=\xi _{xa}$, respectively $%
a\xi _{x}=\xi _{\alpha ^{-1}\left( a\right) x}$, and the inner products are
defined as $\left\langle \xi _{x},\xi _{y}\right\rangle _{A}=x^{\ast }y$,
respectively $_{A}\left\langle \xi _{x},\xi _{y}\right\rangle =\alpha \left(
xy^{\ast }\right) $. The crossed product $A\times _{\alpha }\mathbb{Z}$ of $%
A $ by $\alpha $ is isomorphic to the crossed product of $A$ by $X_{\alpha
}\ $\cite{JZ}.

\begin{corollary}
If $\alpha $ is an inverse limit automorphism of a non unital pro-$C^{\ast }$%
-algebra $A$, then $M(A)\times _{\overline{\alpha }}\mathbb{Z}$ can be
identified with a pro-$C^{\ast }$-subalgebra of $M(A\times _{\alpha }\mathbb{%
Z})$.
\end{corollary}

\textbf{Acknowledgements. }The first author was partially supported by the
grants of the Romanian Ministry of Education, CNCS - UEFISCDI, project
number PN-II-RU-PD-2012-3-0533 and project number PN-II-ID-PCE-2012-4-0201.
The second author was supported by a grant of the Romanian Ministry of
Education, CNCS - UEFISCDI, project number PN-II-RU-PD-2012-3-0533.

\end{document}